\documentclass[11pt,leqno]{article}
\usepackage{amsmath,amssymb,amsthm,nicefrac,mathrsfs}
\usepackage[colorlinks]{hyperref}
\usepackage{tabu}
\usepackage{bbm}
\usepackage{comment}
\usepackage{bigints}
\usepackage{graphics}
\usepackage{verbatim}
\usepackage[a4paper, margin=4cm, footskip=1cm]{geometry}
\usepackage{caption}
\usepackage[T1]{fontenc}
\usepackage{booktabs}
\usepackage{siunitx}
\usepackage{leftidx}
\newtheorem{theorem}{Theorem}[section]
\newtheorem{definition}[theorem]{Definition}
\newtheorem{lemma}[theorem]{Lemma}
\newtheorem{proposition}[theorem]{Proposition}
\newtheorem{corollary}[theorem]{Corollary}

\newtheorem{remark}[theorem]{Remark}

\numberwithin{equation}{section}

\newcommand{\be}{\begin{equation}}
	\newcommand{\ee}{\end{equation}}

\newcommand{\bes}{\begin{equation*}}
	\newcommand{\ees}{\end{equation*}}

\providecommand{\norm}[1]{\Vert#1\Vert}

\renewcommand{\geq}{\geqslant}
\renewcommand{\leq}{\leqslant}

\def\m1{\mathbf{1}}

\allowdisplaybreaks[1]

\allowdisplaybreaks[1]
\author{Ngartelbaye Guerngar\\
	University of North Alabama\\
	\and Erkan Nane\\ Auburn University\\}
\title{Propagation of high peaks for the space-time fractional stochastic partial differential equations

	\date{}
}

\begin{document}
	
	\maketitle
	\begin{abstract}

		We study the space-time nonlinear fractional stochastic heat equation driven by a space-time white noise,
		
		\begin{align*}
			\partial_t^\beta	 u(t,x)=-(-\Delta)^{\alpha/2}u(t,x)+I_t^{1-\beta}\Big[\sigma(u(t,x))\dot{W}(t,x)\Big],\ \ t>0, \ x\in \mathbb{R} , 
		\end{align*}
		where $\sigma:\mathbb{R}\rightarrow\mathbb{R}$ is a globally Lipschitz function and the initial condition is a  measure on $\mathbb{R}.$ Under some growth conditions on $\sigma,$ we derive two important properties about the moments of the solution:
		(i) For $p\geq 2,$ the $p^{\text{th}}$ absolute moment of the solution to the equation above grows exponentially with time.
		(ii) Moreover, the distances to the origin of the farthest high peaks of these moments grow exactly exponentially with time.
		Our results provide an extension of the work of Chen and Dalang \cite{ChenDalang} to a time-fractional setting. We also show that condition (i) holds when we study the same equation for $x\in\mathbb{R}^d.$
	\end{abstract}
	\newpage
	\section{Introduction}

	Since their introduction in the 17$^{th}$ century by Isaac Newton and Gottfried Leibniz, differential equations have been used to model various phenomena. However, differential equations of integer order have shown some limitations in the modeling of more complex phenomena. Recently, fractional differential equations  have gained popularity since they can handle such problems. Their applications include modelling contamination of groundwater flow, the electrical dynamics of the heart and the design of new materials, {\it viscoelasticity} \cite{MainDi}. Stochastic partial differential equations (SPDEs) incorporate randomness in the system \cite{DavCBMS}. They generalize partial differential equations (PDEs).

	In this paper, we study the space-time fractional nonlinear stochastic  equation (see for a motivation to study these type of  equations, and the physical derivation of these equations  in \cite{chen-kim-kim-15} and  \cite{MIJNane-2015})
	\begin{equation}\label{Eq:1}
		\begin{split}
			\begin{cases}
				\partial_t^\beta u(t,x)=&-(-\Delta)^{\alpha/2}u(t,x)+I_t^{1-\beta}\big[\sigma(u(t,x)\dot{W}(t,x)\big],  \ t>0,\  x\in \mathbb{R}, \\
				\ \  u(0,\cdot)=&\mu(\cdot),
			\end{cases} 
		\end{split}
	\end{equation}
	with $\alpha\in(0,2]$,  $\sigma$ is a globally Lipschitz function defined on $\mathbb{R}$, $\dot{W}$ denotes the space-time white noise, $-(-\Delta)^{\alpha/2}$ is the fractional Laplacian. The initial datum $\mu$ is a measure satisfying some conditions (to be specified later). $\partial_t^\beta$ is the Caputo fractional differential operator for $\beta\in(0,1)$, defined by:
	\begin{equation*}
		\partial_t^\beta f(t)=\frac{1}{\Gamma(1-\beta)}\int_0^t \frac{f'(s)}{(t-s)^\beta} ds,
	\end{equation*}
	and $I_t^\gamma$ is the Riemann-Liouville fractional integral of order $\gamma>0$, defined by
	$$I_t^\gamma f(t)=\frac{1}{\Gamma(\gamma)}\int_0^t(t-s)^{\gamma-1}f(s)ds, \ \text{for}\ t>0,$$  
	with the convention that $I_t^0=\text{Id}$, the identity operator.
	
	Following Walsh \cite{Walsh}, we interpret \eqref{Eq:1} in the mild sense, i.e
	\begin{equation}\label{Eq:MildSol}
		\begin{split}
			u(t,x)=&J_0(t,x)+\mathcal{I}(t,x), \text{ where} \\
			J_0(t,x)=&\int\limits_{\mathbb{R}} G(t,x-y)\mu(dy)\ \text{and} \
			\mathcal{I}(t,x)=\int\limits_0^t \int\limits_{\mathbb{R}} \ G(t-s,x-y)\sigma\big(u(s,y)\big)W(ds,dy),
		\end{split}
	\end{equation}
	where $G(\cdot,\cdot)$ is the "heat kernel". It is known that this  kernel satisfies the following bounds for $\alpha\in (0,2)$:
	\begin{equation}\label{G:Bounds}
		c_1\Big(t^{-\beta/\alpha}\wedge \frac{t^\beta}{|x|^{1+\alpha}}\Big)\leq G(t,x)\leq c_2\Big(t^{-\beta/\alpha}\wedge \frac{t^\beta}{|x|^{1+\alpha}}\Big),
	\end{equation}
	for positive constants $c_1$ and $c_2$ and where the upper bound is only valid for $1<\alpha,$ see for example \cite[Lemma 2.1]{FooNane} and the references therein.

	Using a measured-valued  initial data allows us to consider a wide range of initial data including functions with rough paths like the {\it delta Dirac} function.
	Let $\mathcal{M}(\mathbb{R})$ be the set of signed Borel measures on $\mathbb{R}.$ Recall that by the Jordan decomposition, $\mu=\mu_+-\mu_-$, where $\mu_+, \mu_-$ are two nonnegative Borel measures with disjoint support. Then $|\mu|=\mu_++\mu_-.$ It follows that the admissible set for the initial data is 
	$$
	\mathcal{M}_\alpha(\mathbb{R}):=\Big\{\mu\in\mathcal{M}(\mathbb{R}): \sup\limits_{y\in\mathbb{R}}\int_{\mathbb{R}}\frac{1}{1+|x-y|^{1+\alpha}}|\mu|(dx)<\infty\Big\}, \ \text{for} \ \alpha\in (1,2].
	$$
	We will also use the notation $ \mathcal{M}_{\alpha,+}(\mathbb{R}):=\Big\{\mu\in\mathcal{M}_\alpha(\mathbb{R}): \mu \ \text{is nonnegative}\Big\}.$
	
	Eq. \eqref{Eq:1} is of particular interest since it is an extension of equation (1.1) in \cite{ChenLang,ChenDalang}.

	When studying Eq. \eqref{Eq:1}, the second feature of interest, after the existence and uniqueness of the solution, is  the asymptotic properties of the solution since it displays {\it intermittency}. Intermittency has been studied by many authors recently \cite{ChenLang, ChenDalang, ChenHuNua,  fooVar,GuerNane}.

	To describe mathematically this physical property, we define the $p\text{th}$ {\it upper} and {\it lower Lyapunov exponents} of the random field  $u:=\{u(t,x)\}_{t>0, \ x\in \mathbb{R}}$ at $x_0$, respectively as:
	
	\begin{equation}\label{LyapExp}
		\begin{split}
			\overline{\gamma}_p(x_0):=& \limsup\limits_{t\rightarrow\infty}\frac{1}{t}\log\mathbb{E}|u(t,x_0)|^p  \\
			\text{and}\\
			\underline{\gamma}_p(x_0):=& \liminf\limits_{t\rightarrow\infty}\frac{1}{t}\log\mathbb{E}|u(t,x_0)|^p.
		\end{split}
	\end{equation}
	It is known that when the initial datum is constant, the two moments defined in \eqref{LyapExp} do not depend on $x_0.$ In this case, Bertini and Cancrini \cite{BerCan} defined the solution $u$ to be {\it intermittent} if 
	\begin{equation}\label{Interm}
		\gamma_p:\overline{\gamma}_p=\underline{\gamma}_p \ \text{for all}\  p\in\mathbb{N}\  \text{and}\  p\mapsto\frac{\gamma_p}{p} \   \text{is strictly increasing.}   
	\end{equation}
	
	Carmona and Molchanov \cite{CarNov} defined the notion of {\it asymptotic and full intermittency} as follows:  let $p^*=\inf\{n: \gamma_n>0\}$. If $p^*<\infty,$ then the solution $u$ is said to exhibit {\it(asymptotic) intermittency of order $p^*$}. If $p^*=2$, it is said to exhibit {\it full intermittency.} It was also shown that {\it full intermittency} implies the {\it intermittency} defined by Bertini and Cancrini in \eqref{Interm}.
	Also, following \cite{ChenDalang}, the solution is said to be {\it weakly intermittent of type I} if $\underline{\gamma}_2>0$ and {\it weakly intermittent of type II} if $\overline{\gamma}_2>0$. Clearly, the {\it weak intermittency of type I} is stronger than the  {\it weak intermittency of type II} but weaker (expectedly!) than the {\it full intermittency}. Chen et al. \cite{ChenHuNua} and Mijena and Nane \cite{MijNane} showed the weak intermittency of type I for the solution of \eqref{Eq:1}.
	
	This mathematical definition of intermittency is related to the property that the solutions are close to zero in vast regions of space-time but develop high peaks on some small "islands" \cite{ChenLang}.  
	In order to properly characterize the speed of propagation of these "high peaks", we recall the {\it upper and lower growth indices of linear type}, respectively,  see \cite{ChenLang, ConVar} for example for more details about these "moments".
	
	\begin{equation}\label{UpGrwth}
		\overline{\gamma}(p):=\inf\Big\{a>0: \lim\limits_{t\rightarrow\infty}\frac{1}{t}\sup\limits_{|x|\geq at}\log\mathbb{E}|u(t,x)|^p <0\Big\},
	\end{equation}
	and
	\begin{equation}\label{LwGrwth}
		\underline{\gamma}(p):=\sup\Big\{a>0: \lim\limits_{t\rightarrow\infty}\frac{1}{t}\sup\limits_{|x|\geq at}\log\mathbb{E}|u(t,x)|^p >0\Big\}.
	\end{equation}

	These quantities are of interest since they provide information about the possible locations of "high peaks", and how they propagate away from the origin. In fact, if $\underline{\gamma}(p)=\overline{\gamma}(p):=\gamma(p)$, then there will be high peaks at time $t$ inside $\Big[-\gamma(p)t,\ \gamma(p)t\Big]$, but no peaks outside of this interval \cite{ChenLang, MijNane}.
	It is not hard to check  directly that $0 \leq \underline{\gamma}(p) \leq  \overline{\gamma}(p) \leq  \infty$.
	When
	$0< \underline{\gamma}(p) \leq  \overline{\gamma}(p)<  \infty$, it follows that: 
	\begin{enumerate}
		\item[(i)] The solution to Eq. \eqref{Eq:1} has very high peaks as $t\rightarrow\infty$, i.e {\it weak intermittency}; 
		\item[(ii)] The distances between the origin and the farthest high peaks grow exactly linearly in $t$.
	\end{enumerate}

	Conus and Koshnevisan \cite{ConVar} studied Eq. \eqref{Eq:1} with $\beta=1$  and proved that  $0< \underline{\gamma}(p) \leq  \overline{\gamma}(p)<  \infty$ if the initial function $u_0$ is a nonnegative, lower semicontinuous function with compact support of positive Lebesgue measure. In particular, the authors showed that $\frac{\xi^2}{2\pi}\leq\underline{\gamma}(2)\leq \overline{\gamma}(2) \leq \frac{\xi^2}{2}$ if in addition $\alpha=2$ and $\sigma(u)=\xi u$ (Parabolic Anderson Model) in Eq. \eqref{Eq:1}.

	Chen and Dalang \cite{ChenLang}, improved the existence result in \cite{ConVar} by working under a much weaker condition on the initial datum, namely by letting $\mu$ be any signed Borel measure over $\mathbb{R}$ such that 
	$
	\int\limits_{\mathbb{R}}e^{-ax^2}|\mu|(dx)<\infty \ \text{for all } \ a>0.
	$
	The authors also showed that for the Parabolic Anderson Model in this case, $\underline{\xi}(2)=\overline{\xi}(2).$

	Chen and Dalang \cite{ChenDalang} considered Eq. \eqref{Eq:1} with $\beta=1$ and allowed the operator $-(-\Delta)^{\alpha/2}$ to have some positive skewness $\delta$. The authors defined the following growth indices of exponential type:

	\begin{equation}\label{UpExpGrwth}
		\overline{\xi}(p):=\inf\Big\{a>0: \lim\limits_{t\rightarrow\infty}\frac{1}{t}\sup\limits_{|x|\geq \exp{(at)}}\log\mathbb{E}|u(t,x)|^p <0\Big\},
	\end{equation}
	and
	\begin{equation}\label{LwExpGrwth}
		\underline{\xi}(p):=\sup\Big\{a>0: \lim\limits_{t\rightarrow\infty}\frac{1}{t}\sup\limits_{|x|\geq \exp{(at)}}\log\mathbb{E}|u(t,x)|^p >0\Big\}.
	\end{equation}

	In this case, the authors showed that 
	for  the quantities in \eqref{UpExpGrwth} and \eqref{LwExpGrwth}, $0< \underline{\xi}(p) \leq  \overline{\xi}(p)<  \infty$ if the initial datum has sufficiently rapid decay at $\pm\infty$.
	
	In this paper, we assume that $\alpha\in(1,2)$ in Eq. \eqref{Eq:1}, i.e  the underlying process has both positive and negative jumps. We show that for all $p\geq 2$,  $\overline{\xi}(p)<\infty$  if the initial datum has sufficiently rapid decay at $\pm\infty$, see Corollary \ref{Cor:1}. On the other hand, we also prove that if  $\mu\in\mathcal{M}_{a,+}(\mathbb{R}), \ \mu\neq 0$  then $\underline{\xi}(p)>0$ for all $p\geq 2$ and provided that the growth condition \eqref{eq:1.3} (specified below) is satisfied, see Theorem \ref{Thm:3.4}. Our results provide an extension of the results in \cite{ChenDalang}( with $\delta=0 $).

	Next, we assume that the function $\sigma:\mathbb{R}\rightarrow\mathbb{R}$ is globally Lipschitz with Lipschitz constant $\text{Lip}_\sigma>0.$ As aforementioned, the following growth conditions are  needed in our calculations: assume there are some constants $l_\sigma, L_\sigma, \underline{\varsigma}$ and $\overline{\varsigma}$ such that
	
	\begin{equation}\label{eq:1.3}
		\sigma(x)^2\geq l_\sigma^2(\underline{\varsigma}^2+x^2), \ \text{for all} \ x\in \mathbb{R}.
	\end{equation}
	and 
	\begin{equation}\label{eq:upG}
		\sigma(x)^2\leq L_\sigma^2(\overline{\varsigma}^2+x^2), \ \text{for all} \ x\in \mathbb{R}.
	\end{equation}
	Note that these growth conditions are direct consequences of the Lipschitz continuity of the function $\sigma.$
	
	The rest of the article is structured as follows:  we state our main results in section \ref{sect2}.In section \ref{sect3} we state and  prove some technical results needed for  the proofs of our main results. We then prove our main results  in section \ref{sect4}. In section \ref{sect5}, we state and prove extension of our results to higher spatial dimensions $(d>1)$.  The article concludes with an appendix in section \ref{Appdx} where some useful results from other authors are compiled. Throughout this paper, the letter c in upper or lower case, with or without a subscript or superscript is a constant whose exact value may not be of great importance for our results. Also, $"\star"$ represents the convolution in both the time and space variables, while $"*"$ represents the convolution in a single variable, for example the space (or time) variable.
	
	\section{Main Results}\label{sect2}
	The following Theorems and corollary extend \cite[Theorem 3.6]{ChenDalang} with $\delta=0$  to the corresponding space-time fractional stochastic heat equation. All these results are proven in   section \ref{sect4}.
	\begin{theorem}\label{thm36}
		Suppose $1<\alpha<2$ and $\sigma$ satisfies Eq. \eqref{eq:upG} with $\bar{\varsigma}=0.$ If 
		\begin{equation}\label{eq:eta}
			\int\limits_{\mathbb{R}}|\mu|(dy)(1+|y|^\eta)<\infty
		\end{equation}
		for some $\eta>0,$ then 
		there are positive constants $C<\infty$ and $b=\min(\eta, 2)$ such that for all $(t,x)\in[1,\infty)\times \mathbb{R}$,
		\begin{equation}\label{eq:J0}
			|J_{0}(t,x)|\leq	|J_{\alpha,\beta}(t,x)|\leq C\big(1+t^{\beta/\alpha}\big)(1+|x|)^{-b}, 
		\end{equation} 
		
		where $J_{\alpha,\beta}:=\overline{\mathcal{G}}_{\alpha,\beta}\star\mu $  and  $\overline{\mathcal{G}}_{\alpha,\beta}$ is defined in \eqref{RefKern}. 
	\end{theorem}

	\begin{corollary}\label{Cor:1}
		Assume the conditions of Theorem \ref{thm36} hold, then 
		\begin{equation}\label{eq:ep}
			\bar{\xi}(p)\leq \frac{c_2}{b}<\infty. 
		\end{equation}
		for some positive constant $c_2.$
	\end{corollary}
	
	\begin{theorem}\label{Thm:3.4}
		Suppose that $\alpha\in (1,2]$ and $\sigma$ satisfy the growth condition \eqref{eq:1.3}. For all $\mu\in \mathcal{M}_{\alpha, +}(\mathbb{R}), \mu\neq 0$ and for all $p\geq 2,$ if $\underline{\varsigma}=0,$ then 
		\begin{equation}
			\underline{\xi}(p)\geq \frac{\Psi^{\frac{1}{1-\beta/\alpha}}}{2(\alpha+1)}>0,
		\end{equation}
		where $\Psi$ is constant depending on $\alpha,\beta$, and $l_\sigma.$ 
		
		For these $\mu,$ if $\underline{\varsigma}=0$ and $\mu(dx)=f(x)dx$ with $f(x)\geq c$ for all $x\in \mathbb{R}$ or if $\underline{\varsigma}\neq 0$, then $\underline{\xi}(p)=\bar{\xi}(p)=+\infty.$ 
		In particular, $\underline{\gamma}(p)=\bar{\gamma}(p)=+\infty.$  
	\end{theorem}

	\section{Preliminaries}\label{sect3}
	Let $W=\big\{W_t(A): A\in \mathcal{B}_b(\mathbb{R}), t\geq 0\big\}$ be a space-time white noise defined on a complete probability space $(\Omega, \mathcal{F}, P),$ where $\mathcal{B}_b(\mathbb{R})$ is the collection of Borel sets with finite Lebesgue measure. Let $\mathcal{F}_t=\sigma\Big(W_s(A): 0\leq s\leq t, A \in \mathcal{B}_b(\mathbb{R})\Big)\vee\mathcal{N}, \ t\geq 0,$ be the natural filtration augmented by the $\sigma$-field generated by all $P$-null sets in $\mathcal{F}.$ Then $W$ is martingale measure, and the integral $\iint\limits_{[0,t]\times\mathbb{R}}X(s,y)W(ds,dy)$ is well-defined in the Walsh sense \cite{Walsh} for a suitable class of random fields $\big\{X(s,y), (s,y)\in\mathbb{R}_+\times\mathbb{R}\big\}$. We will also use the notation $\norm{\cdot}_p$ for the $L^p(\Omega)-$norm for $p\geq 1.$
	
	\vspace{0.5em}
	The following definition provides an interpretation of the solution to our main problem.
	\begin{definition}\label{Def-1}
		Following \cite{ChenDalang}, a random filed $u:=\{u(t,x)\}_{t>0, \ x\in \mathbb{R}}$ is called a {\it  solution} of \eqref{Eq:1} if the following conditions hold:
		\begin{enumerate}
			\item $u$ is adapted, i.e, for all $(t,x)\in \mathbb{R}^*_+\times \mathbb{R},$ $u(t,x)\in \mathcal{F}_t;$
			\item $u$ is jointly measurable with respect  to $\mathcal{B}(\mathbb{R}^*_+\times \mathbb{R})\times \mathcal{F}$;
			\item for all $(t,x)\in \mathbb{R}^*_+\times \mathbb{R},$ the following space-time convolution is finite:\\
			$\Big(G^2\star\norm{\sigma(u)}_2^2\Big)(t,x):=\int\limits_0^t ds\int\limits_{\mathbb{R}} dy \ G^2(t-s,x-y)\norm{\sigma\big(u(s,y)\big)}_2^2<\infty.$
			\item the function 
			$\mathcal{I}: \ \mathbb{R}^*_+\times \mathbb{R}\rightarrow L^2(\Omega)$        
			is continuous;
			\item $u$ satisfies \eqref{Eq:MildSol}
			for all $(t,x)\in \mathbb{R}^*_+\times \mathbb{R}.$
		\end{enumerate}
	\end{definition}
	
	\vspace{0.25cm}
	The existencee and uniqueness  of such solution are all proved in \cite[Theorem 3.2]{ChenHuNua}, therefore we do not replicate these results here. Instead we focus on proving that the inequality, $0< \underline{\xi}(p) \leq  \overline{\xi}(p)<  \infty$, holds. To this aim, we need to introduce some kernel functions.
	For all $(t,x)\in \mathbb{R}^*_+\times \mathbb{R}, n\in\mathbb{N}$ and $\lambda\in\mathbb{R},$  define,
	$$
	\mathcal{L}_0(t,x):=G^2(t,x),
	$$
	\begin{equation}\label{LnConv}
		\mathcal{L}_n(t,x):=(\underbrace{\mathcal{L}_0\star\cdots \star\mathcal{L}_0}_{n \ \text{factors of}\ \mathcal{L}_0}) (t,x), \ \text{for} \ n\geq 1,
	\end{equation}
	\begin{center}
		and,
	\end{center}
	\begin{equation}\label{Eq:K}
		\mathcal{K}(t,x;\lambda):=\sum\limits_{n=0}^\infty\lambda^{2(n+1)}\mathcal{L}_n(t,x).
	\end{equation}
	The following variations of the kernel functions $\mathcal{K}$ will also be used.
	\begin{equation}\label{EqKs}
		\begin{split}
			&\mathcal{K}(t,x):=\mathcal{K}(t,x;\lambda) \qquad \overline{\mathcal{K}}(t,x):=\mathcal{K}(t,x;L_\sigma),  \\
			& \underline{\mathcal{K}}(t,x):=\mathcal{K}(t,x;l_\sigma) \qquad \hat{\mathcal{K}}(t,x):=\mathcal{K}(t,x;4\sqrt{p}L_\sigma), \ \text{for} \ p\geq 2. \\
		\end{split}
	\end{equation}

	The moment bounds of the solution of Eq. \eqref{Eq:1}, see \cite[(3.3) and (3.4)]{ChenHuNua}, depend heavily on the kernels defined above in \eqref{EqKs}. To use them, we need good estimates on the kernel function $\mathcal{K}.$ Due to the fractional time-derivative, the "heat kernel" $G$ lacks several important properties including  the Chapman-Kolmogorov identity (semigroup property). We therefore introduce additional kernel functions that display the semi-group property or at least some restrictive form of it (like sub-semigroup property or sup-semigroup property). Define the {\it reference kernel functions:}
	\begin{equation}\label{RefKernAlp}
		\underline{\mathcal{G}}_{\alpha,\beta}(t,x):=\frac{{C}_{1,\alpha}t^{\beta}}{\big(t^{2\beta/\alpha}+x^2\big)^{\frac{1+\alpha}{2}}} \qquad \text{for} \ \alpha\in(0,2)\  \text{and}\  \beta\in(0,1)
	\end{equation}
	\begin{center}
		and,
	\end{center} 
	\begin{equation}\label{RefKern}
		\overline{ \mathcal{G}}_{\alpha,\beta}(t,x):=\frac{\overline{C}_1t^{\beta/\alpha}}{t^{2\beta/\alpha}+x^2} \qquad \text{for} \  \alpha\in(0,2)\  \text{and}\  \beta\in(0,1),
	\end{equation}
	
	where ${C}_{1,\alpha}$ and $\overline{C}_1$ are understood to be  normalizing constants. 
	
	For example, \begin{equation}\label{NormConst}
		C_{1,\alpha}:=\frac{\Gamma(\alpha/2+1/2)}{\Gamma(\alpha/2)\Gamma(1/2)} \ \ \text{and} \ \overline{C}_1:=\frac{1}{\pi}.
	\end{equation}

	It is  easy to verify directly that, for $1\leq \alpha,$ 
	\begin{equation}\label{KerInq1}
		G(t,x)\leq \underline{\mathcal{G}}_{\alpha,\beta}(t,x)\leq \frac{C_{1,\alpha}}{\overline{C}_1} \overline{\mathcal{G}}_{\alpha,\beta}(t,x).
	\end{equation}

	\vspace{0.5em}
	The proofs of our main results require some technical results and will be provided in Section \ref{sect4}. The next four results are needed to prove Theorem \ref{Thm:3.4}.
	\begin{lemma}\label{lm51}
		For all $x\in\mathbb{R}$ and $t>0,$
		\begin{equation}\label{lwbG}
			G(t,x)\geq \tilde{C}\underline{\mathcal{G}}_{\alpha,\beta}(t,x)
		\end{equation}
		for some positive constant $\tilde{C}.$
	\end{lemma}
	\begin{proof}
		Using the scaling property of $G$ and $\underline{\mathcal{G}}_{\alpha,\beta},$ we have
		\begin{center}	
			$\inf\limits_{(t,x)\in\mathbb{R}_+^*\times\mathbb{R}^d}\frac{G(t,x)}{\underline{\mathcal{G}}_{\alpha,\beta}(t,x)}=\inf\limits_{y\in\mathbb{R}}\frac{G(1,y)}{\underline{\mathcal{G}}_{\alpha,\beta}(1,y)}.$
		\end{center}
		
		Recall from \eqref{G:Bounds} that $G(t,x)\geq c_1\Big(t^{-\beta/\alpha}\wedge \frac{t^\beta}{|x|^{1+\alpha}}\Big)$ for some positive constant $c_1.$ We consider two cases.
		
		Case 1: $|y|>1$. In this case, 
		\begin{align*}
			\frac{G(1,y)}{\underline{\mathcal{G}}_{\alpha,\beta}(1,y)}\geq \tilde{C}_1&\frac{(1+|y|^2)^{\frac{1+\alpha}{2}}}{|y|^{1+\alpha}} \geq \Tilde{C}_2.
		\end{align*}
		
		Case 2: $|y|\leq 1$. Then $\frac{G(1,y)}{\underline{\mathcal{G}}_{\alpha,\beta}(1,y)}\geq \Tilde{c}_1\big(1+|y|^2\big)^{\frac{1+\alpha}{2}}\geq \Tilde{c}_2.$
		
		It follows that in both cases, we have $\inf\limits_{y\in\mathbb{R}}\frac{G(1,y)}{\underline{\mathcal{G}}_{\alpha,\beta}(1,y)}>0$ and this concludes the proof.
	\end{proof}

	The next Lemma provides some lower bound estimates on the reference kernel function $\underline{\mathcal{G}}_{\alpha,\beta}.$ They will be used in the proof of our second main result.
	\begin{lemma}\label{lm:1.4}
		\begin{enumerate}
			\item For all $t>0,$  $\underline{\mathcal{G}}_{\alpha,\beta}(t,x-y)\geq 2^{-(1+\alpha)}{C}_{1,\alpha}^{-1}t^{\beta /\alpha} \underline{\mathcal{G}}_{\alpha,\beta}(t,x)\underline{\mathcal{G}}_{\alpha,\beta}(t,y).$
			\item For $t>0$ and $z\in\mathbb{R},$  $\mathcal{F}\big[\underline{\mathcal{G}}_{\alpha,\beta}^2(t,\cdot)\big](z)\geq {C}_{1,\alpha}^2 C_{\alpha+1/2}t^{-\beta /\alpha} \exp\big(-t^{\beta/\alpha}|z|\big),$ where $C_{\alpha+1/2}$ is the constant defined in Lemma \ref{lm52}.
			\item For all $ t\geq s>0$ and $x\in\mathbb{R},$ we have
			
			$\Big(\underline{\mathcal{G}}_{\alpha,\beta}^2(t-s,\cdot)*\underline{\mathcal{G}}_{\alpha,\beta}^2(s,\cdot)\Big)(x)\geq \frac{C_{1,\alpha}^2C_{\alpha+1/2}^2}{2^{2\alpha+3}\pi }(st)^{-\beta /\alpha}(t-s)^{\beta/\alpha}\underline{\mathcal{G}}_{\alpha,\beta}^2(t-s,x)$.
			\item For $t\geq r\geq t/2>0,$ we have $\underline{\mathcal{G}}_{\alpha,\beta}(r,x)\geq 2^{-\beta(1+1/\alpha)}(t/r)^{\beta/\alpha}\underline{\mathcal{G}}_{\alpha,\beta}(t,x)$.
		\end{enumerate}
	\end{lemma}
	
	\begin{proof}
		\begin{enumerate}
			\item Note that for all $a,b\in\mathbb{R}, 1+(a-b)^2\leq 1+2a^2+2b^2\leq (1+2a^2)(1+2b^2). $ Thus
			\begin{align*}
				\underline{\mathcal{G}}_{\alpha,\beta}(t,x-y)=&\frac{C_{1,\alpha}t^{\beta}}{\big[t^{2\beta/\alpha}+(x-y)^2\big]^{(1+\alpha)/2}}\\
				=&\frac{C_{1,\alpha}t^{-\beta /\alpha}}{\Big(1+\big[t^{-\beta/\alpha}(x-y)\big]^2\Big)^{(1+\alpha)/2}}\\
				\geq &\frac{C_{1,\alpha}t^{-\beta /\alpha}}{\Bigg(\Big[1+\big(t^{-\beta/\alpha}\sqrt{2}x\big)^2\Big]\Big[1+\big(t^{-\beta/\alpha}\sqrt{2}y\big)^2\Big]\Bigg)^{(1+\alpha)/2}}\\
				=& C_{1,\alpha}^{-1} t^{\beta /\alpha} \underline{\mathcal{G}}_{\alpha,\beta}(t,\sqrt{2}x)\underline{\mathcal{G}}_{\alpha,\beta}(t,\sqrt{2}y).
			\end{align*}
			Now 
			\begin{align*}
				\underline{\mathcal{G}}_{\alpha,\beta}(t,\sqrt{2}x)=& \frac{C_{1,\alpha}t^{-\beta /\alpha}}{2^{\frac{1+\alpha}{2}}\Big[2^{-1}+\big(t^{-\beta/\alpha}x\big)^2\Big]^{(1+\alpha)/2}}\\
				&\geq 2^{-\frac{1+\alpha}{2}} \ \underline{\mathcal{G}}_{\alpha,\beta}(t,x).
			\end{align*}
			Therefore,
			\[\underline{\mathcal{G}}_{\alpha,\beta}(t,x-y)\geq C_{1,\alpha}^{-1}2^{-(1+\alpha)}t^{\beta /\alpha}\underline{\mathcal{G}}_{\alpha,\beta}(t,x)\underline{\mathcal{G}}_{\alpha,\beta}(t,y).\]	
			\item Apply Lemma \ref{lm52} with $\nu=\alpha+1/2$ and $b=t^{\beta/\alpha}.$
			\item By  part 1. of this Lemma, we have
			\begin{align*}
				\Big(\underline{\mathcal{G}}_{\alpha,\beta}^2(t-s,\cdot)*\underline{\mathcal{G}}_{\alpha,\beta}^2(s,\cdot)\Big)(x)\geq C_{1,\alpha}^{-2}2^{-2(1+\alpha)}(t-s)^{2\beta /\alpha}\underline{\mathcal{G}}_{\alpha,\beta}^2(t-s,x)\int\limits_{\mathbb{R}}\underline{\mathcal{G}}_{\alpha,\beta}^2(t-s,y)\underline{\mathcal{G}}_{\alpha,\beta}^2(s,y)dy.
			\end{align*}
			Now using Plancherel's identity and part 2. of this Lemma, we have,
			\begin{align*}
				\int\limits_{\mathbb{R}}\underline{\mathcal{G}}_{\alpha,\beta}^2(t-s,y)\underline{\mathcal{G}}_{\alpha,\beta}^2(s,y)dy\geq& \frac{C_{1,\alpha}^4C_{\alpha+1/2}^2}{2\pi}\int\limits_{\mathbb{R}}\Big[s(t-s)\Big]^{-\beta/\alpha}\exp\Big[-\big((t-s)^{\beta/\alpha}+s^{\beta/\alpha}\big)|z|\Big]dz\\
				=&  \frac{C_{1,\alpha}^4C_{\alpha+1/2}^2}{2\pi}\big[s(t-s)\big]^{-\beta/\alpha}\frac{2}{(t-s)^{\beta/\alpha}+s^{\beta/\alpha}}\\
				\geq&  \frac{C_{1,\alpha}^4C_{\alpha+1/2}^2}{2\pi}\big[s(t-s)\big]^{-\beta/\alpha}t^{-\beta /\alpha}.
			\end{align*}
			\item For $t\geq r\geq t/2>0,$  we have 
			\begin{align*}
				\underline{\mathcal{G}}_{\alpha,\beta}(r,x) =&C_{1,\alpha} r^{-\beta/\alpha}\Bigg(1+\frac{x^2}{r^{2\beta/\alpha}}\Bigg)^{-{(1+\alpha)/2}}\\
				\geq &\frac{C_{1,\alpha} r^{-\beta/\alpha}t^{\beta(1+1/\alpha)}}{2^{\beta(1+1/\alpha)}}\Big((t/2)^{2\beta/\alpha}+x^2\Big)^{-(1+\alpha)/2}\\
				\geq &\frac{1}{2^{\beta(1+1/\alpha)}} r^{-\beta/\alpha}t^{\beta/\alpha} \underline{\mathcal{G}}_{\alpha,\beta}(t,x).
			\end{align*}
			This concludes the proof.
		\end{enumerate}
	\end{proof}
	
	\vspace{0.5cm}

	\begin{lemma}\label{lm:5.4}
		Suppose $\alpha\in (1,2]$ and $\mu\in \mathcal{M}_{\alpha, +}(\mathbb{R}), \ \mu\neq 0$. Then for all $\epsilon>0,$ there exists a positive constant $C_\#:=C_\#(\alpha,\beta)$ such that for all $t\geq 0$ and $x\in \mathbb{R},$
		\begin{equation}\label{lwb:J0}
			J_0(t,x)=\Big(G(t, \cdot)*\mu\Big)(x)\geq C_{\#}\boldsymbol{1}_{\{t\geq \epsilon\}}	\underline{\mathcal{G}}_{\alpha,\beta}(t,x).
		\end{equation}
	\end{lemma}

	\begin{proof}
		Combining Lemma \ref{lm51} and Lemma \ref{lm:1.4} (1), we get
		\begin{align*}
			J_0(t,x)\geq& \tilde{C} \int\limits_{\mathbb{R}} \underline{\mathcal{G}}_{\alpha,\beta}(t,x-y) \mu(dy)\\
			\geq & \tilde{C}_1t^{\beta/\alpha}\underline{\mathcal{G}}_{\alpha,\beta}(t,x)\int\limits_{\mathbb{R}} \underline{\mathcal{G}}_{\alpha,\beta}(t,y) \mu(dy)\\
			=& \tilde{C}_2\underline{\mathcal{G}}_{\alpha,\beta}(t,x)\bigintss\limits_{\mathbb{R}} \Bigg(1+\frac{y^2}{t^{2\beta/\alpha}}\Bigg)^{-(1+\alpha)/2} \mu(dy).
		\end{align*}
		Note that the integrand above is non-decreasing with respect with $t$. Therefore, 
		\begin{align*}
			J_0(t,x)\geq& \tilde{C}_3\boldsymbol{1}_{\{t\geq \epsilon\}}(t)\underline{\mathcal{G}}_{\alpha,\beta}(t,x)\bigintss\limits_{\mathbb{R}} \Bigg(1+\frac{y^2}{\epsilon^{2\beta/\alpha}}\Bigg)^{-(1+\alpha)/2} \mu(dy)\\
			=&\tilde{C}_4\boldsymbol{1}_{\{t\geq \epsilon\}}(t)\epsilon^{\beta/\alpha}\underline{\mathcal{G}}_{\alpha,\beta}(t,x)\int\limits_{\mathbb{R}} \underline{\mathcal{G}}_{\alpha,\beta}(\epsilon, y) \mu(dy).
		\end{align*}
		Since the function $y\mapsto \underline{\mathcal{G}}_{\alpha,\beta}(\epsilon, y) $ is strictly positive and by the choice of $\mu$, it is clear that the integral above is positive. Thus, the relation \eqref{lwb:J0} follows by taking $$C_{\#}:=\tilde{C}_4\epsilon^{\beta/\alpha}\int\limits_{\mathbb{R}} \underline{\mathcal{G}}_{\alpha,\beta}(\epsilon, y) \mu(dy).$$
	\end{proof}

	\begin{proposition}\label{Prop:3.3}
		Let $\alpha\in(1,2].$ Then for all $t>0$ and $x\in \mathbb{R},$
		\begin{equation}\label{Eq:3.19}
			\mathcal{K}(t,x)\geq C_\star \underline{\mathcal{G}}_{\alpha,\beta}^2(t,x)\text{E}_{1-\frac{\beta}{\alpha},1-\frac{\beta}{\alpha}}\Big[\Psi t^{1-\frac{\beta}{\alpha}}\Big],
		\end{equation}
		where $\Psi:=\Psi(\alpha,\beta, \lambda)>0.$
		
		\vspace{0.25cm}
		In particular, for all $t>0$ and $x\in\mathbb{R},$
		\begin{equation}\label{Eq:3.20}
			\big(1\star \mathcal{K}\big)(t,x)\geq C_\circ t^{1-\frac{\beta}{\alpha}} \text{E}_{1-\frac{\beta}{\alpha},2-\frac{\beta}{\alpha}}\Big[\Psi t^{1-\frac{\beta}{\alpha}}\Big].
		\end{equation}
		Here, $C_*,$ and $C_o$ are all positive constants depending on $\alpha$ and $\beta.$
	\end{proposition}

	\begin{proof}
		Denote the $n-$fold convolution product 
		$$
		\Big(\underline{\mathcal{G}}_{\alpha,\beta}^2\Big)^{\star n}(t,x):=\Big(\underbrace{\underline{\mathcal{G}}_{\alpha,\beta}^2\star \cdots \star \underline{\mathcal{G}}_{\alpha,\beta}^2}_{n \ \text{factors of} \ \underline{\mathcal{G}}_{\alpha,\beta}^2}\Big)(t,x).
		$$
		By the definition of $\mathcal{K}$ and the estimate \eqref{lwbG}, 
		\begin{equation}\label{Eq:5.4}
			\mathcal{K}(t,x;\lambda)=\sum\limits_{n=0}^\infty\Big(\lambda^2G^2\Big)^{\star (n+1)}(t,x)\geq \sum\limits_{n=0}^\infty\Big(\lambda^2\Tilde{C}^2\underline{\mathcal{G}}_{\alpha,\beta}^2\Big)^{\star (n+1)}(t,x).  
		\end{equation}
		We now find a lower bound for the term $\Big(\lambda^2\underline{\mathcal{G}}_{\alpha,\beta}^2\Big)^{\star (n+1)}(t,x)$.
		
		\vspace{0.5em}
		\textbf{\underline{Claim}:} 
		\begin{equation}\label{Eq:5.5} \Big(\lambda^2\underline{\mathcal{G}}_{\alpha,\beta}^2\Big)^{\star (n+1)}(t,x)\geq \frac{\Xi_{\alpha,\beta}^{n}\lambda^{2(n+1)}\Gamma(1-\beta/\alpha)^{n+1}}{\Gamma\big((n+1)(1-\beta/\alpha)\big)}t^{n(1-\beta/\alpha)}\underline{\mathcal{G}}_{\alpha,\beta}^2(t,x) \ \ \ \ \text{for all} \ n\geq 0,
		\end{equation}
		where $\Xi_{\alpha,\beta}:=2^{-2\big(\alpha+\beta(1+1/\alpha)+2\big)}\pi^{-1}C_{1,\alpha}^2C_{\alpha+1/2}^2$ is a positive constant.

		\vspace{0.25cm}
		We proceed by induction. The case $n=0$ is straightforward. So consider $n\geq 1$ and assume by induction that \eqref{Eq:5.5} holds for $n-1.$ Combining the induction hypothesis with Lemma \ref{lm:1.4} (3), we get
		
		\begin{align*}
			\Big(\lambda^2\underline{\mathcal{G}}_{\alpha,\beta}^2\Big)^{\star (n+1)}(t,x)&\geq  \frac{\Xi_{\alpha,\beta}^{n-1}\lambda^{2(n+1)}\Gamma(1-\beta/\alpha)^{n}}{\Gamma\big(n(1-\beta/\alpha)\big)}\int\limits_0^t (t-s)^{(n-1)(1-\beta/\alpha)}\Big(\underline{\mathcal{G}}_{\alpha,\beta}^2(t-s,\cdot)*\underline{\mathcal{G}}_{\alpha,\beta}^2(s,\cdot) \Big)(x)ds\\
			&\geq \Upsilon_n t^{-\beta/\alpha}\int\limits_0^t \underline{\mathcal{G}}_{\alpha,\beta}^2(t-s,x)  (t-s)^{(n-1)(1-\beta/\alpha)+2\beta/\alpha}\Big[s(t-s)\Big]^{-\beta/\alpha}ds,
		\end{align*}
		where 
		$$\Upsilon_n:=\frac{\Xi_{\alpha,\beta}^{n-1}\lambda^{2(n+1)}\Gamma(1-\beta/\alpha)^{n}}{\Gamma\big(n(1-\beta/\alpha)\big)}\cdot\frac{C_1^2C_{\alpha+1/2}^2}{2^{2\alpha+3}\pi}.$$
		Now, if $0\leq s\leq t/2$, then $t-s\geq t/2$. Thus, we can apply Lemma \ref{lm:1.4} (4) to see that
		\begin{align*}
			\Big(\lambda^2\underline{\mathcal{G}}_{\alpha,\beta}^2\Big)^{\star (n+1)}(t,x)\geq \Upsilon_n2^{-2\beta(1+1/\alpha)}t^{\beta/\alpha}\underline{\mathcal{G}}_{\alpha,\beta}^2(t,x)\int\limits_0^{t/2}    (t-s)^{(n-1)(1-\beta/\alpha)}\Big[s(t-s)\Big]^{-\beta/\alpha}ds. 
		\end{align*}
		Since $t-s\geq s$ for $0\leq s\leq t/2$, we get
		\begin{align*}
			\Big(\lambda^2\underline{\mathcal{G}}_{\alpha,\beta}^2\Big)^{\star (n+1)}(t,x)\geq \Upsilon_n2^{-2\beta(1+1/\alpha)}t^{\beta/\alpha}\underline{\mathcal{G}}_{\alpha,\beta}^2(t,x)\int\limits_0^{t/2}   s^{(n-1)(1-\beta/\alpha)}\Big[s(t-s)\Big]^{-\beta/\alpha}ds.
		\end{align*}
		Thus,
		\begin{align*}
			\Big(\lambda^2\underline{\mathcal{G}}_{\alpha,\beta}^2\Big)^{\star (n+1)}(t,x)\geq \Upsilon_n2^{-2\beta(1+1/\alpha)-1}t^{\beta/\alpha}\underline{\mathcal{G}}_{\alpha,\beta}^2(t,x)\int\limits_0^t    s^{(n-1)(1-\beta/\alpha)}\Big[s(t-s)\Big]^{-\beta/\alpha} ds. 
		\end{align*}
		Finally, applying the definition of \textit{Euler's Beta integral,} see Eq. \eqref{EulBta}, proves the claim.
		
		\vspace{0.25cm}
		Therefore, 
		\begin{align*}
			\mathcal{K}(t,x;\lambda)&\geq \Tilde{C}^2\lambda^2\Gamma(1-\beta/\alpha)\underline{\mathcal{G}}_{\alpha,\beta}^2(t,x)\sum\limits_{n=0}^\infty\frac{\Big(\Xi_{\alpha,\beta} \Tilde{C}^2\lambda^2\Gamma\big(1-\beta/\alpha\big)t^{1-\beta/\alpha}\Big)^n}{\Gamma\big((n+1)(1-\beta/\alpha)\big)} \\
			&\geq \Tilde{C}^2\lambda^2\Gamma(1-\beta/\alpha)\underline{\mathcal{G}}_{\alpha,\beta}^2(t,x)\text{E}_{1-\frac{\beta }{\alpha},\ 1-\frac{\beta }{\alpha}}\Big(\Xi_{\alpha,\beta}\lambda^2 \Tilde{C}^2\Gamma(1-\beta/\alpha)t^{1-\beta/\alpha}\Big).
		\end{align*}

		Thus the estimate \eqref{Eq:3.19} follows by setting 
		\begin{equation}\label{Psi}
			\Psi:= \Xi_{\alpha,\beta} \Tilde{C}^2\lambda^2\Gamma(1-\beta/\alpha) \ \ 
			\text{and} \ \
			C_*:=\lambda^2\Gamma(1-\beta/\alpha)\Tilde{C}^2.
		\end{equation}

		\vspace{0.25cm}
		As for \eqref{Eq:3.20}, using \eqref{Eq:3.19}, we get
		\begin{align*}
			\Big(1\star\mathcal{K}\Big)(t,x)=& \int\limits_0^t \int\limits_{\mathbb{R}}  \mathcal{K}(s,y) dy ds\\
			\geq & C\int\limits_0^t  \text{E}_{1-\frac{\beta}{\alpha},\ 1-\frac{\beta}{\alpha}}\Big(\Psi s^{1-\beta/\alpha}\Big)\int\limits_{\mathbb{R}} \underline{\mathcal{G}}_{\alpha,\beta}^2(s,y)dyds. 
		\end{align*}
		Now,
		\begin{align*}
			\int\limits_{\mathbb{R}} \underline{\mathcal{G}}_{\alpha,\beta}^2(s,y)dy&=C_{1,\alpha}^2 s^{-\beta/\alpha} 2\int\limits_0^\infty \frac{1}{(1+z^2)^{(1+\alpha)}} dz\\
			&=C_{2,\alpha} s^{-\beta/\alpha}.
		\end{align*}
		Finally, using Eq. \eqref{IntMtgLf}, we conclude the proof.
	\end{proof}
	\vspace{0.5em}
	
	The next lemma is crucial in the proof of Corollary \ref{Cor:1}.
	\begin{lemma}\label{lm410}
		Suppose $1<\alpha\leq 2$ and $\mu\in \mathcal{M}(\mathbb{R})$. Then
		\begin{enumerate}
			\item $J_{0}(t,x)=\Big(G(t,\cdot)*\mu\Big)(x)\in C^\infty\big(\mathbb{R_+}^*\times \mathbb{R}\big)$.
			\item For all compact sets $K\subset \mathbb{R_+}^*\times \mathbb{R} $ and $\nu\in\mathbb{R}$,
			\begin{equation}\label{Eq:Sup}
				\sup\limits_{(t,x)\in K}\Big(\big[\nu^2+J_0^2\big]\star \mathcal{K}\Big)(t,x)<\infty.
			\end{equation}
			In fact, for all $(t,x)\in \mathbb{R_+}^*\times \mathbb{R},$
			\begin{equation}\label{Eq:J0}
				\Big(J_0^2\star\mathcal{K}\Big)(t,x)\leq C(t\vee 1)^{\beta}t^{1-\beta/\alpha}\big[t^{-\beta/\alpha}+e^{c_2t}\big]\big|J_{\alpha.\beta}(t,x)\big|
			\end{equation}
			where $J_{\alpha,\beta}:=\overline{\mathcal{G}}_{\alpha,\beta}*\mu $ and for some positive constants $C$ and $c_2.$
		\end{enumerate}
	\end{lemma}
	\begin{proof}
		Part 1. and \eqref{Eq:Sup} follow from \cite{ChenHuNua}. So we only provide the proof of of the estimate \eqref{Eq:J0}.
		
		Recall from \cite[Theorem 3.4 (1)]{ChenHuNua} that 
	
		\begin{align*}
			\mathcal{K}(t,x;\lambda)\leq C_1 \overline{\mathcal{G}}_{\alpha,\beta}(t,x)\Big(t^{-\beta/\alpha}+e^{c_2 t}\Big),
		\end{align*}
		and from \cite[P. 5103]{ChenHuNua} that 
		\begin{equation}\label{J0}
			J_0(s,y)\leq  C_2 s^{-\beta/\alpha}(1\vee t)^\beta, \ \text{for} \ s\in(0,t].
		\end{equation}
		
		It follows that 
		
		\begin{align*}
			\Big(J_0^2\star\mathcal{K}\Big)(t,x)\leq & C_3\int_0^t ds\Big[(t-s)^{-\beta/\alpha}+e^{c_2(t-s)}\Big]\int\limits_{\mathbb{R}}dy \ \overline{\mathcal{G}}_{\alpha,\beta}(t-s,x-y)\\
			&\qquad\qquad\qquad\times s^{-\beta/\alpha}(1\vee t)^\beta\Bigg|\int\limits_{\mathbb{R}}
			\mu(dz){G}(s,y-z)\Bigg|.
		\end{align*}
		Now, integrating over $dy$, combining the estimate \eqref{KerInq1} with the sub-semigroup property of $\overline{\mathcal{G}}_{\alpha,\beta}$, see the relation (5.14) in \cite{ChenHuNua} for example,
		and at last integrating over $\mu(dy)$, we get:
		\begin{align*}
			\Big(J_0^2\star\mathcal{K}\Big)(t,x)\leq & C_4 (t\vee 1)^{\beta}\big|J_{\alpha,\beta}(t,x)\big|\int_0^t  s^{-\beta/\alpha}\Big[(t-s)^{-\beta/\alpha}+e^{c_2(t-s)}\Big] ds\\
			\leq & C_5 (t\vee 1)^\beta  t^{1-\beta/\alpha}\Big(t^{-\beta/\alpha}+e^{c_2t}\Big)\big|J_{\alpha,\beta}(t,x)\big|,
		\end{align*}
		where $C_5$ is a positive constant depending on $ \alpha$ and $\beta$.
	\end{proof}

	\vspace{0.25cm}	
	
	With all the necessary tools at our disposal, we are now ready to prove our main results.

	\section{Proofs of  the main results}\label{sect4}
	
	\begin{proof}[Proof of Theorem \ref{thm36}]
		The proof follows similar ideas from \cite{ChenDalang}. Assume \eqref{eq:eta} holds for some $\eta>0.$ We consider two cases here.
		
		\vspace{0.25cm}
		\textbf{Case 1:} $\boldsymbol{0<\eta<2.}$
		\vspace{0.25cm}
		
		Using \eqref{RefKern}, we have
		\begin{align*}
			|J_{\alpha,\beta}(t,x)|\leq& \bigintssss\limits_{\mathbb{R}}\frac{C_1(1+t^{\beta/\alpha})}{1+|x-y|^2}|\mu|(dy)\\
			\leq & C_2 (1+t^{\beta/\alpha})\sup\limits_{y\in\mathbb{R}}\Big[(1+|y|)\big(1+|x-y|\big)^{2/\eta}\Big]^{-\eta}.
		\end{align*}
		Note that $\upsilon=2/\eta>1 $ and 
		$(1+|x-y|^\upsilon)(1+|y|)\geq 1+|x-y|^\upsilon+|y|.$ Thus, using Lemma \ref{lm5.5}, it follows that 
		\begin{center}
			$|J_{\alpha,\beta}(t,x)|\leq C_3(1+t^{\beta/\alpha})\frac{1}{1+|x|^\eta}$.
		\end{center}
		
		\vspace{0.25cm}
		\textbf{Case 2:} $\boldsymbol{\eta\geq 2.}$
		
		\vspace{0.25cm}
		Observe that 
		\begin{equation*}
			|J_{\alpha,\beta}(t,x)|\leq \bigintsss\limits_{\mathbb{R}}\frac{\overline{\mathcal{G}}_{\alpha,\beta}(t,x-y)\big(1+|x-y|^2\big)}{1+|x-y|^2}|\mu|(dy).
		\end{equation*}
		Next, it is not hard to see from Eq. \eqref{RefKern}  that 
		\begin{center}
			$\overline{\mathcal{G}}_{\alpha,\beta}(t,x-y)\big(1+|x-y|^2\big)\leq C t^{\beta/\alpha}.$
		\end{center}
		Now, set $\omega=\eta/2\geq 1.$ Note that 
		\begin{align*}
			\Big(1+|x-y|^2\Big)\Big(1+|y|^{2\omega}\Big)\geq& \frac{1}{2}\Big(1+|x-y|^2+|y|^{2\omega}\Big)\\
			\geq & c_\omega \Big(1+|x-y|^2+|y|^{2}\Big)\\
			\geq & C_\omega\Big(1+|x|^2\Big)
		\end{align*}
		for some constant $C_\omega>0.$ It follows that for all $t>1$ and $x\in\mathbb{R},$
		\begin{align*}
			|J_{\alpha,\beta}(t,x)|\leq& C_4 t^{\beta/\alpha} \bigintssss\limits_{\mathbb{R}}\frac{(1+|y|^\eta)}{(1+|x-y|^2)(1+|y|^{2\omega})}|\mu|(dy)\\
			\leq& C_5 \frac{t^{\beta/\alpha}}{(1+|x|^2)^{(1+\alpha)/2}}\int\limits_{\mathbb{R}}(1+|y|^\eta)|\mu|(dy)\\
			\leq & C_6 \frac{1+t^{\beta/\alpha}}{(1+|x|)^2}.
		\end{align*}
		This concludes the proof.
	\end{proof}

	\begin{proof}[Proof of Corollary \ref{Cor:1}]
		Pick $p\geq 2.$  Using Lemma \ref{lm410} with $t\geq 1$ yields
		$$\Big(J_0^2\star\mathcal{\hat{K}}\Big)(t,x)\leq C t^{1+\beta(1-1/\alpha)}  \Big(t^{-\beta/\alpha}+e^{c_2t}\Big)\big|J_{\alpha, \beta}(t,x)\big|.$$
		For $a, b> 0$ and using \eqref{eq:J0}, we have  
		\begin{align*}
			\lim\limits_{t\rightarrow \infty}\frac{1}{t}\sup\limits_{|x|\geq \exp(at)}\log{\norm{u(t,x)}}_p^2=&\lim\limits_{t\rightarrow \infty}\frac{1}{t}\sup\limits_{|x|\geq \exp(at)}\log\Big(J_0^2\star\mathcal{\hat{K}}\Big)(t,x)\\
			\leq & c_2-ab.
		\end{align*}
		Now, $c_2-ab<0$ if and only if $a>\frac{c_2}{b}$. Thus,
		\begin{equation}
			\bar{\xi}(p):=\inf\Big\{a>0: \lim\limits_{t\rightarrow \infty}\frac{1}{t}\sup\limits_{|x|\geq \exp(at)}\log{\mathbb{E}{\big|u(t,x)\big|^p}}<0\Big\}\leq \frac{c_2}{b}<\infty.
		\end{equation}
	\end{proof}

	\begin{proof}[Proof of Theorem \ref{Thm:3.4}]
		Since $\underline{\xi}(p)\geq \underline{\xi}(2)$ for all $p\geq 2$,  it is enough to consider only the case $p=2.$

		To this aim, fix $\epsilon\in(0,t/2)$ and choose a positive constant $C_{\#}$ as in Lemma \ref{lm:5.4} such that 
		$$J_0(t,x)\geq C_{\#}\boldsymbol{1}_{\{t\geq \epsilon\}}	\underline{\mathcal{G}}_{\alpha,\beta}(t,x):=I_\epsilon(t,x).$$
		By \cite[(3.4)]{ChenHuNua},
		$$\norm{u(t,x)}_2^2\geq J_0^2(t,x)+\Big(J_0^2\star\underline{\mathcal{K}}\Big)(t,x)\geq \Big(I_\epsilon^2\star\underline{\mathcal{K}}\Big)(t,x).$$
		Now, applying Proposition \ref{Prop:3.3} and Lemma \ref{lm:1.4} (3), we get
		\begin{align*}
			\Big(I_\epsilon^2\star\underline{\mathcal{K}}\Big)(t,x)&\geq \hat{C}_1\int\limits_0^{t-\epsilon} ds \ \text{E}_{1-\frac{\beta}{\alpha}, \ 1-\frac{\beta}{\alpha}}\big(\Psi s^{1-\beta/\alpha}\big)\int\limits_{\mathbb{R}} dy \ \underline{\mathcal{G}}_{\alpha,\beta}^2(t-s,x-y)\underline{\mathcal{G}}_{\alpha,\beta}^2(s,y)\\
			&\geq \hat{C}_2t^{-\beta/\alpha}\int\limits_0^{t-\epsilon} \text{E}_{1-\frac{\beta}{\alpha}, \ 1-\frac{\beta}{\alpha}}\big(\Psi s^{1-\beta/\alpha}\big)s^{-\beta/\alpha}(t-s)^{\beta/\alpha}\underline{\mathcal{G}}_{\alpha,\beta}^2(t-s,x)ds.
		\end{align*}
		Since $\underline{\mathcal{G}}_{\alpha,\beta}(t-s,x)\geq \big(\frac{t-s}{t}\big)^\beta\underline{\mathcal{G}}_{\alpha,\beta}(t,x)$, it follows that
		\begin{align*}
			\Big(I_\epsilon^2\star\underline{\mathcal{K}}\Big)(t,x)&\geq \hat{C}_3t^{-\beta(2+1/\alpha)}\underline{\mathcal{G}}_{\alpha,\beta}^2(t,x)\int\limits_0^{t-\epsilon} \text{E}_{1-\frac{\beta}{\alpha}, \ 1-\frac{\beta}{\alpha}}\big(\Psi s^{1-\beta/\alpha}\big)s^{-\beta/\alpha}(t-s)^{\beta(2+1/\alpha)}ds\\
			&\geq \hat{C}_4t^{-\beta(2+1/\alpha)}\underline{\mathcal{G}}_{\alpha,\beta}^2(t,x) \text{E}_{1-\frac{\beta}{\alpha}, \ 1-\frac{\beta}{\alpha}}\Big(\Psi (t-2\epsilon)^{1-\beta/\alpha}\Big)
			\int\limits_{t-2\epsilon}^{t-\epsilon}  s^{-\beta/\alpha}(t-s)^{\beta(2+1/\alpha)}ds\\
			&\geq  \hat{C}_5t^{-\beta(2+1/\alpha)}\underline{\mathcal{G}}_{\alpha,\beta}^2(t,x) \text{E}_{1-\frac{\beta}{\alpha}, \ 1-\frac{\beta}{\alpha}}\Big(\Psi (t-2\epsilon)^{1-\beta/\alpha}\Big) \frac{\epsilon^{\beta(2+1/\alpha)}}{(t-\epsilon)^{\beta/\alpha}}\epsilon\\
			& =  C_{\epsilon}t^{-\beta(2+1/\alpha)}(t-\epsilon)^{-\beta/\alpha}\underline{\mathcal{G}}_{\alpha,\beta}^2(t,x)\text{E}_{1-\frac{\beta}{\alpha}, \ 1-\frac{\beta}{\alpha}}\Big(\Psi (t-2\epsilon)^{1-\beta/\alpha}\Big).
		\end{align*}
		Clearly, the function $x\mapsto\underline{\mathcal{G}}_{\alpha,\beta}(t,x)$ is even and decreasing for $x\geq 0.$ It follows that for all $a>0,$
		\begin{align*}
			\sup\limits_{|x|>\exp{(at)}}\norm{u(t,x)}_2^2\geq  C_{\epsilon} \underline{\mathcal{G}}^2_{\alpha,\beta}\big(t,\exp{(at)}\big)t^{-2\beta-\beta/\alpha}(t-\epsilon)^{-\beta/\alpha}\text{E}_{1-\frac{\beta}{\alpha}, \ 1-\frac{\beta}{\alpha}}\Big(\Psi (t-2\epsilon)^{1-\beta/\alpha}\Big).
		\end{align*}
		Since $\alpha, \beta>0,$ there exists some $t_0\geq 0$ such that for all $t\geq t_0,$  $t^{\beta/\alpha}\leq e^{at},$ so
		$$\underline{\mathcal{G}}_{\alpha,\beta}^2\big(t,\exp{(at)}\big)\geq \frac{C_\alpha t^{2\beta}}{e^{2a(\alpha+1)t}}.$$
		Finally, the asymptotic expansion of the Mittag-Leffler function in Lemma \ref{lm4.2} shows that
		
		\begin{equation}\label{Eq:5.9}
			\lim\limits_{t\rightarrow\infty}\frac{1}{t}\sup\limits_{|x|\geq \exp{(at)}}\log\norm{u(t,x}_2^2\geq \Psi^{\frac{1}{1-\beta/\alpha}}-2a(\alpha+1).
		\end{equation}
		Therefore,
		\begin{align*}
			\underline{\xi}(2)=&\sup\Big\{a: \lim\limits_{t\rightarrow\infty}\frac{1}{t}\sup\limits_{|x|\geq \exp{(at)}}\log\norm{u(t,x}_2^2>0 \Big\}\\
			\geq & \sup\Big\{a>0: \Psi^{\frac{1}{1-\beta/\alpha}}-2a(\alpha+1)>0\Big\}\\
			=& \frac{\Psi^{\frac{1}{1-\beta/\alpha}}}{2(\alpha+1)}.
		\end{align*}
		As for the second part of the Theorem, suppose that $\underline{\varsigma}=0$ and that there is $c>0$ such that $J_0\geq c,$ or that $\underline{\varsigma}\neq 0.$ In this case, by \cite[(3.4)]{ChenHuNua} and Proposition \ref{Prop:3.3}, we have 
		$$\norm{u(t,x)}_2^2\geq \max(c^2,\underline{\varsigma}^2)\big(1\star\underline{\mathcal{K}}\big)(t,x)\geq \bar{C}t^{1-\beta/\alpha}\text{E}_{1-\frac{\beta}{\alpha}, \ 2-\frac{\beta}{\alpha}}\big(\Psi t^{1-\beta/\alpha}\big).$$
		Note that the lower bound above is independent of $x$, thus by Lemma \ref{lm4.2},

		\begin{equation}\label{Eq:5.10}
			\lim\limits_{t\rightarrow\infty}\frac{1}{t}\sup\limits_{|x|\geq \exp{(at)}}\log\norm{u(t,x}_2^2\geq \Psi^{\frac{1}{1-\beta/\alpha}}.
		\end{equation}
		
		Thus, $\underline{\xi}(2)=+\infty.$ This concludes the proof.

	\end{proof}
	We now extend our results to higher spatial dimensions.	
	\section{Higher Dimension}\label{sect5}
	
	In this section,   we  consider the following equation 
	\begin{equation}\label{Eq:1-d}
		\begin{split}
			\begin{cases}
				\partial_t^\beta u(t,x)=&-(-\Delta)^{\alpha/2}u(t,x)+I_t^{1-\beta}\big[\sigma(u(t,x)\dot{W}(t,x)\big],  \ t>0,\  x\in \mathbb{R}^d, \\
				\ \  u(0,\cdot)=&\mu(\cdot),
			\end{cases} 
		\end{split}
	\end{equation}
	with everything else defined as in Eq.\eqref{Eq:1}.
	
	Mijena and Nane in \cite{MIJNane-2015} studied this equation and showed that a necessary and sufficient condition for the existence and uniqueness of the solution is 
	\begin{equation}\label{cond-higher-spde}
		d< \alpha \min \{ \beta^{-1},2\}.
	\end{equation}
	Note that this condition implies that  for some values of $\alpha$ and $\beta$ the equation has a unique solution in dimensions $d=1,2$ and $3$, see \cite{MIJNane-2015}.

	Moreover, intermittency and  intermittency fronts $(\text{when}\ \alpha=2)$  properties were studied by Asogwa and Nane in  \cite{Asogwa-nane-2017}, Foondun and Nane in \cite{FooNane}, and Mijena and Nane \cite{MijNane}.
	
	For  the higher spatial dimension, we need to update the admissible set for the initial datum as 
	$$
	\mathcal{M}_{\alpha,d}(\mathbb{R}^d):=\Big\{\mu\in\mathcal{M}(\mathbb{R}^d): \sup\limits_{y\in\mathbb{R}^d}\int_{\mathbb{R}^d}\frac{1}{1+|x-y|^{d+\alpha}}|\mu|(dx)<\infty\Big\}, \ \text{for} \ \alpha\in (1,2].
	$$

	We again interpret \eqref{Eq:1-d} in the mild sense, i.e
	\begin{equation}\label{Eq:MildSol-d}
		\begin{split}
			u(t,x)=&J_0(t,x)+\mathcal{I}_d(t,x), \text{ where} \\
			J_0(t,x)=&\int\limits_{\mathbb{R}^d} G_d(t,x-y)\mu(dy)\ \text{and} \
			\mathcal{I}_d(t,x)=\int\limits_0^t \int\limits_{\mathbb{R}^d} \ G_d(t-s,x-y)\sigma\big(u(s,y)\big)W(ds,dy).
		\end{split}
	\end{equation}
	The "heat kernel" $ G_d(\cdot,\cdot)$ now satisfies an updated version of \eqref{G:Bounds}, i.e for $\alpha\in (0,2)$:
	\begin{equation}\label{G:Bounds-}
		c_1\Big(t^{-\beta d/\alpha}\wedge \frac{t^\beta}{|x|^{d+\alpha}}\Big)\leq G_d(t,x)\leq c_2\Big(t^{-\beta d/\alpha}\wedge \frac{t^\beta}{|x|^{d+\alpha}}\Big),
	\end{equation}
	for positive constants $c_1$ and $c_2$ and where again the upper bound is only valid for $d=1<\alpha,$ see for example \cite[Lemma 2.1]{FooNane} and the references therein.
	
	Similar to the case $d=1,$ we also provide a definition of the solution to our problem, Eq \eqref{Eq:1-d}.
	\begin{definition}
		Following \cite{ChenDalang}, a random filed $u:=\{u(t,x)\}_{t>0, \ x\in \mathbb{R}^d}$ is called a {\it  solution} of \eqref{Eq:1-d} if the following conditions hold:
		\begin{enumerate}
			\item $u$ is adapted, i.e, for all $(t,x)\in \mathbb{R}^*_+\times \mathbb{R}^d,$ $u(t,x)\in \mathcal{F}_t;$
			\item $u$ is jointly measurable with respect  to $\mathcal{B}(\mathbb{R}^*_+\times \mathbb{R}^d)\times \mathcal{F}$;
			\item for all $(t,x)\in \mathbb{R}^*_+\times \mathbb{R}^d,$ the following space-time convolution is finite:\\
			$\Big(G^2_d\star\norm{\sigma(u)}_2^2\Big)(t,x):=\int\limits_0^t ds\int\limits_{\mathbb{R}^d} dy \ G^2_d(t-s,x-y)\norm{\sigma\big(u(s,y)\big)}_2^2<\infty.$
			\item the function 
			$\mathcal{I}_d: \ \mathbb{R}^*_+\times \mathbb{R}^d\rightarrow L^2(\Omega)$        
			is continuous;
			\item $u$ satisfies \eqref{Eq:MildSol-d}
			for all $(t,x)\in \mathbb{R}^*_+\times \mathbb{R}^d.$
		\end{enumerate}
	\end{definition}

	Define the following reference kernel
	
	\begin{equation}\label{HighDenst}
		\leftidx{_d}{\mathcal{G}}{_{\alpha,\beta}}(t,x)=\frac{C_{d,\alpha}t^\beta}{\big(t^{2\beta/\alpha}+|x|^2\big)^{(d+\alpha)/2}}, \ t>0\ \text{and} \ x\in \mathbb{R}^d, \ \text{with} \ C_{d,\alpha}:=\frac{\Gamma(d/2+\alpha/2)}{\pi^{d/2}\Gamma(\alpha/2)}.
	\end{equation}
	
	Here, $|x-y|=\sqrt{(x_1-y_1)^2+(x_2-y_2)^2+\cdots+(x_d-y_d)^2}$ represents the Euclidean distance between two points $x,y\in\mathbb{R}^d$. We will also use the differential $dx:=dx_1dx_2\cdots dx_d.$
	
	\vspace{0.25cm}

	Note that for all $\zeta>0$ and $x\in\mathbb{R}^d$, we have
	\begin{align*}
		1+\zeta^{-2}|x|^2\leq \big(1+(\zeta^{-1}x_1)^2)\big(1+(\zeta^{-1}x_2)^2)\cdots\big(1+(\zeta^{-1}x_d)^2).
	\end{align*}
	Therefore,
	\begin{align}
		\leftidx{_d}{\mathcal{G}}{_{\alpha,\beta}}(t,x)
		&=\frac{C_{d,\alpha}t^{-\beta d/\alpha}}{\big(1+t^{-2\beta/\alpha}|x|^2\big)^{(d+\alpha)/2}}\nonumber\\
		&\geq C_{d,\alpha} \prod_{k=1}^d \frac{ t^{-\beta /\alpha}}{\big(1+(t^{-\beta/\alpha}x_k)^2\big)^{(d+\alpha)/2}} \nonumber \\
		&= C_{d,\alpha} \prod_{k=1}^d \frac{ t^{\frac{\beta}{\alpha}(d-1)+\beta}}{\big(t^{2\beta/\alpha}+x_k^2\big)^{(d+\alpha)/2}} \label{ineq-d-to-one}
	\end{align}
	

	\vspace{0.5cm}

	\begin{lemma}\label{lwb-d}
		The following lower bound estimates hold for all $x,$ and $y\in \mathbb{R}^d$:
		\begin{enumerate}
			\item For all $t>0,  \ \leftidx{_d}{\mathcal{G}}{_{\alpha,\beta}}(t,x-y)\geq 2^{-(d+\alpha)}t^{\frac{\beta d}{\alpha}}\leftidx{_d}{\mathcal{G}}{_{\alpha,\beta}}(t,x)\leftidx{_d}{\mathcal{G}}{_{\alpha,\beta}}(t,y).$
			\item For all $t\geq s>0,$ \\ $\Big(\leftidx{_d}{\mathcal{G}}{^2_{\alpha,\beta}}(t-s,\cdot)*\leftidx{_d}{\mathcal{G}}{^2_{\alpha,\beta}}(s,\cdot)\Big)(x)\geq \frac{C_{d,\alpha}^2C_{\alpha+1/2}^{2d}}{2^{3d+2\alpha}\pi^d}(st)^{-\frac{\beta d}{\alpha}}(t-s)^{\frac{\beta d}{\alpha}}\leftidx{_d}{\mathcal{G}}{^2_{\alpha,\beta}}(t-s,x).$ 
			\item For all $t\geq r\geq t/2>0,$  $\leftidx{_d}{\mathcal{G}}{_{\alpha,\beta}}(r,x)\geq 2^{-\beta(1+d/\alpha)}(t/r)^{\frac{\beta d}{\alpha}}\leftidx{_d}{\mathcal{G}}{_{\alpha,\beta}}(t,x)$.
		\end{enumerate}
	\end{lemma}
	\begin{proof}
		The proof follows similar ideas from the one dimensional case in Lemma \ref{lm:1.4}. We provide the outlines below.
		
		\begin{enumerate}
			\item Observe that for all $a,b\in\mathbb{R}^d, 1+|a-b|^2\leq 1+2|a|^2+2|b|^2\leq (1+2|a|^2)(1+2|b|^2). $ 
			
			Therefore,
			\begin{align*}
				\leftidx{_d}{\mathcal{G}}{_{\alpha,\beta}}(t,x-y)=&\frac{C_{d,\alpha}t^{\beta}}{\big(t^{2\beta/\alpha}+|x-y|^2\big)^{(d+\alpha)/2}}\\
				\geq & C_{d,\alpha}^{-1} t^{\frac{\beta d}{\alpha}} \leftidx{_d}{\mathcal{G}}{_{\alpha,\beta}}(t,\sqrt{2}x)\leftidx{_d}{\mathcal{G}}{_{\alpha,\beta}}(t,\sqrt{2}y).
			\end{align*}
			Now, 
			\begin{align*}
				\leftidx{_d}{\mathcal{G}}{_{\alpha,\beta}}(t,\sqrt{2}x)
				&\geq 2^{-\frac{d+\alpha}{2}} \ \underline{\mathcal{G}}_{\alpha,\beta}(t,x).
			\end{align*}
			It follows that,
			\[\leftidx{_d}{\mathcal{G}}{_{\alpha,\beta}}(t,x-y)\geq 2^{-(d+\alpha)}C_{d,\alpha}^{-1}t^{\frac{\beta d}{\alpha}}\leftidx{_d}{\mathcal{G}}{_{\alpha,\beta}}(t,x)\leftidx{_d}{\mathcal{G}}{_{\alpha,\beta}}(t,y).\]	
			\item By  part 1. of this Lemma, we have
			\begin{align*}
				\Big(\leftidx{_d}{\mathcal{G}}{^2_{\alpha,\beta}}(t-s,\cdot)*\leftidx{_d}{\mathcal{G}}{^2_{\alpha,\beta}}(s,\cdot)\Big)(x)\geq \frac{(t-s)^{\frac{2\beta d}{\alpha}}}{C_{d,\alpha}^{2}2^{2(d+\alpha)}}\leftidx{_d}{\mathcal{G}}{^2_{\alpha,\beta}}(t-s,x)\int\limits_{\mathbb{R}}\leftidx{_d}{\mathcal{G}}{^2_{\alpha,\beta}}(t-s,y)\leftidx{_d}{\mathcal{G}}{^2_{\alpha,\beta}}(s,y)dy.
			\end{align*}
			Now, using \eqref{ineq-d-to-one}, we have
			\begin{align*}
				&\int\limits_{\mathbb{R}^d}\leftidx{_d}{\mathcal{G}}{^2_{\alpha,\beta}}(t-s,y)\leftidx{_d}{\mathcal{G}}{^2_{\alpha,\beta}}(s,y)dy\nonumber \\
				& \qquad\geq \int\limits_{\mathbb{R}^d} \Bigg(C_{d,\alpha} \prod_{k=1}^d \frac{ (t-s)^{\frac{\beta}{\alpha}(d-1)+\beta}}{\big[(t-s)^{2\beta/\alpha}+y_k^2\big]^{(d+\alpha)/2}} C_{d,\alpha}\prod_{k=1}^d \frac{ s^{\frac{\beta}{\alpha}(d-1)+\beta}}{\big(s^{2\beta/\alpha}+y_k^2\big)^{(d+\alpha)/2}}\Bigg)^2 dy\nonumber \\
				&\qquad = C_{d,\alpha}^4 \big[s(t-s)\big]^{\frac{2\beta d}{\alpha}(d-1)+2\beta d}\prod_{k=1}^d  \int\limits_{\mathbb{R}}\frac{ 1}{\big[(t-s)^{2\beta/\alpha}+y_k^2\big]^{(d+\alpha)}}  \frac{1}{\big(s^{2\beta/\alpha}+y_k^2\big)^{(d+\alpha)}} dy_k  \nonumber \\
				&\qquad\geq     \leftidx{_d}{\Theta}{_{\alpha,\beta}} \big[s(t-s)\big]^{\frac{2\beta d}{\alpha}(d-1)+2\beta d}\big[s(t-s)\big]^{\frac{\beta d}{\alpha}(-2d-2\alpha+1)}\prod\limits_{k=1}^d\int\limits_{\mathbb{R}}\exp\Big[-\Big((t-s)^{\beta/\alpha}+s^{\beta/\alpha}\Big)|z_k|\Big] dz_k \\ 
				&\qquad=  \leftidx{_d}{\Theta}{_{\alpha,\beta}}\big[s(t-s)\big]^{-\frac{\beta d}{\alpha}}\Bigg(\frac{2}{(t-s)^{\beta/\alpha}+s^{\beta/\alpha}}\Bigg)^d \nonumber \\
				&\qquad\geq  \leftidx{_d}{\Theta}{_{\alpha,\beta}}\Big[s(t-s)\Big]^{-\frac{\beta d}{\alpha}}t^{-\frac{\beta d}{\alpha}}.\nonumber
			\end{align*}
			
			Here, $ \leftidx{_d}{\Theta}{_{\alpha,\beta}}:=\frac{C_{d,\alpha}^4 C_{d+ \alpha-1/2}^{2d}}{(2\pi)^d}$ and we have used  Plancherel identity  and  Lemma \ref{lm52} to get the second inequality
			with  $\nu=(d+\alpha)-1/2$.
			\item For $t\geq r\geq t/2>0,$  we have 
			\begin{align*}
				\leftidx{_d}{\mathcal{G}}{_{\alpha,\beta}}(r,x) =&C_{d,\alpha} r^{-\frac{\beta d}{\alpha}}\Bigg(1+\frac{|x|^2}{r^{2\beta/\alpha}}\Bigg)^{-{(d+\alpha)/2}}\\
				\geq &\frac{C_{d,\alpha} r^{-\frac{\beta d}{\alpha}}t^{\beta(1+d/\alpha)}}{2^{\beta(1+d/\alpha)}}\Big((t/2)^{2\beta/\alpha}+|x|^2\Big)^{-(d+\alpha)/2}\\
				\geq &2^{-\beta(1+d/\alpha)} (t/r)^{\frac{\beta d}{\alpha}}\leftidx{_d}{\mathcal{G}}{^2_{\alpha,\beta}}(t,x).
			\end{align*}
			This concludes the proof.
			
		\end{enumerate}
	\end{proof}

	The proof of the next Lemma is similar to that of Lemma \ref{lm51}, so we skip it.
	\begin{lemma}\label{lm51-d}
		For all $x\in\mathbb{R}^d$ and $t>0,$
		\begin{equation}\label{lwbG-d}
			G_d(t,x)\geq \tilde{C}_d\underline{\mathcal{G}}_{\alpha,\beta, d}(t,x)
		\end{equation}
		for some positive constant $\tilde{C}_d.$
	\end{lemma}
	
	\begin{proposition}\label{Prop:3.3-d}
		Let $\alpha\in(1,2].$ Then for all $t>0$ and $x\in \mathbb{R},$ 
		\begin{equation}\label{Eq:3.19d}
			\mathcal{K}_d(t,x)\geq \leftidx{_d}{C}{_\star}\
			\leftidx{_d}{\mathcal{G}}{^2_{\alpha,\beta}}(t,x)\text{E}_{1-\frac{\beta d}{\alpha},1-\frac{\beta d}{\alpha}}\Big[\Psi_dt^{1-\frac{\beta d}{\alpha}}\Big],
		\end{equation}

		\vspace{0.25cm}
		In particular, for all $t>0$ and $x\in\mathbb{R}^d,$
		\begin{equation}\label{Eq:3.20d}
			\big(1\star \mathcal{K}_d\big)(t,x)\geq \leftidx{_d}{C}{_\circ}\ t^{1-\frac{\beta}{\alpha}} \text{E}_{1-\frac{\beta d}{\alpha},2-\frac{\beta d}{\alpha}}\Big[\Psi_dt^{1-\frac{\beta d}{\alpha}}\Big].
		\end{equation}
		Here, $\leftidx{_d}{C}{_\star},\  \leftidx{_d}{C}{_\circ}$ and $\Psi_d$ are all positive constants depending on $d, \alpha$ and $\beta.$
	\end{proposition}

	\begin{proof}
		The proof is similar to that of Proposition \ref{Prop:3.3}. We provide highlights below.

		Denote the convolution product 
		$$
		\Big(\leftidx{_d}{\mathcal{G}}{^2_{\alpha,\beta}}\Big)^{\star n}(t,x):=\Big(\underbrace{\leftidx{_d}{\mathcal{G}}{^2_{\alpha,\beta}}\star \cdots \star \leftidx{_d}{\mathcal{G}}{^2_{\alpha,\beta}}}_{n \ \text{factors of} \ \leftidx{_d}{\mathcal{G}}{^2_{\alpha,\beta}}}\Big)(t,x).
		$$
		By the definition of $\mathcal{K}_d$ and the estimate \eqref{lwbG-d}, 
		\begin{equation}\label{Eq:5.4}
			\mathcal{K}_d(t,x;\lambda)=\sum\limits_{n=0}^\infty\Big(\lambda^2G^2_d\Big)^{\star (n+1)}(t,x)\geq \sum\limits_{n=0}^\infty\Big(\Tilde{C}_d^2\lambda^2\leftidx{_d}{\mathcal{G}}{^2_{\alpha,\beta}}\Big)^{\star (n+1)}(t,x).  
		\end{equation}
		The next step is bound from below the term $\Big(\lambda^2\leftidx{_d}{\mathcal{G}}{^2_{\alpha,\beta}}\Big)^{\star (n+1)}(t,x)$. To this aim, we make the following claim:
		
		\vspace{0.5em}
		\underline{Claim}: 
		\begin{equation}\label{Eq:5.5d} \Big(\lambda^2\leftidx{_d}{\mathcal{G}}{^2_{\alpha,\beta}}\Big)^{\star (n+1)}(t,x)\geq \frac{\lambda^{2(n+1)}\leftidx{_d}{\Xi}{^n_{\alpha,\beta}}\Gamma\Big(1-\frac{\beta d}{\alpha}\Big)^{n+1}}{\Gamma\Big((n+1)\big(1-\frac{\beta d}{\alpha}\big)\Big)}t^{n\big(1-\frac{\beta d}{\alpha}\big)}\leftidx{_d}{\mathcal{G}}{^2_{\alpha,\beta}}(t,x) \ \ \text{for all} \ n\geq 0,
		\end{equation}
		where $\leftidx{_d}{\Xi}{_{\alpha,\beta}}:=2^{-\big(3d+2\alpha+2\beta(1+d/\alpha)+1\big)}\pi^{-d}C_{d,\alpha}^2C_{d+\alpha-1/2}^{2d}$ is a positive constant.

		\vspace{0.25cm}
		Again, the case $n=0$ is straightforward. So we consider $n\geq 1$ and assume by induction that \eqref{Eq:5.5d} holds for $n-1.$ Combining the induction hypothesis with Lemma \ref{lwb-d} (2), we get
		
		\begin{align*}
			\Big(\lambda^2\leftidx{_d}{\mathcal{G}}{^2_{\alpha,\beta}}\Big)^{\star (n+1)}(t,x)
			&\geq \leftidx{_d}{\Upsilon}{_n} t^{-\frac{\beta d}{\alpha}}\int\limits_0^t \ \leftidx{_d}{\mathcal{G}}{^2_{\alpha,\beta}}(t-s,x)  (t-s)^{(n-1)(1-\frac{\beta d}{\alpha})+\frac{2\beta d}{\alpha}}\Big[s(t-s)\Big]^{-\frac{\beta d}{\alpha}} ds,
		\end{align*}
		where 
		$$\leftidx{_d}{\Upsilon}{_n}:=\frac{\leftidx{_d}{\Xi}{_{\alpha,\beta}^{n-1}}\lambda^{2(n+1)}\Gamma(1-\frac{\beta d}{\alpha})^{n}}{\Gamma\big(n(1-\frac{\beta d}{\alpha})\big)}\cdot\frac{C_{d,\alpha}^2C_{d\alpha-1/2}^2}{2^{2\alpha+3}\pi^d}.$$
		Observe that  $t-s\geq t/2$ whenever $0\leq s\leq t/2$. Thus, Lemma \ref{lwb-d} (3) implies that
		\begin{align*}
			\Big(\lambda^2\leftidx{_d}{\mathcal{G}}{^2_{\alpha,\beta}}\Big)^{\star (n+1)}(t,x)\geq \leftidx{_d}{\Upsilon}{_n}2^{-2\beta(1+d/\alpha)}t^{\frac{\beta d}{\alpha}}\leftidx{_d}{\mathcal{G}}{^2_{\alpha,\beta}}(t,x)\int\limits_0^{t/2}    (t-s)^{(n-1)(1-\frac{\beta d}{\alpha})}\Big[s(t-s)\Big]^{-\frac{\beta d}{\alpha}} ds. 
		\end{align*}
		Also, $t-s\geq s$ for $0\leq s\leq t/2$, so
		\begin{align*}
			\Big(\lambda^2\leftidx{_d}{\mathcal{G}}{^2_{\alpha,\beta}}\Big)^{\star (n+1)}(t,x)\geq \leftidx{_d}{\Upsilon}{_n}2^{-2\beta(1+d/\alpha)}t^{\frac{\beta d}{\alpha}}\leftidx{_d}{\mathcal{G}}{^2_{\alpha,\beta}}(t,x)\int\limits_0^{t/2} \   s^{(n-1)\big(1-\frac{\beta d}{\alpha}\big)}\big[s(t-s)\big]^{-\frac{\beta d}{\alpha}} ds.
		\end{align*}
		
		It follows that
		\begin{align*}
			\Big(\lambda^2\leftidx{_d}{\mathcal{G}}{^2_{\alpha,\beta}}\Big)^{\star (n+1)}(t,x)\geq \leftidx{_d}{\Upsilon}{_n}2^{-2\beta(1+d/\alpha)-1}t^{\frac{\beta d}{\alpha}}\leftidx{_d}{\mathcal{G}}{^2_{\alpha,\beta}}(t,x)\int\limits_0^t   s^{(n-1)\big(1-\frac{\beta d}{\alpha}\big)}\big[s(t-s)\big]^{-\frac{\beta d}{\alpha}}ds. 
		\end{align*}
		Finally, applying the definition of Euler's Beta integral, see Eq. \eqref{EulBta}, proves the claim.
		
		\vspace{0.25cm}
		Therefore, 
		\begin{align*}
			\mathcal{K}_d(t,x;\lambda)&\geq \Tilde{C}^2_d\lambda^2\Gamma\Big(1-\frac{\beta d}{\alpha}\Big)\leftidx{_d}{\mathcal{G}}{^2_{\alpha,\beta}}(t,x)\sum\limits_{n=0}^\infty\frac{\Big(\leftidx{_d}{\Xi}{_{\alpha,\beta}} \Tilde{C}^2_d\lambda^2\Gamma\Big(1-\frac{\beta d}{\alpha}\Big)t^{1-\frac{\beta d}{\alpha}}\Big)^n}{\Gamma\Big((n+1)\big(1-\frac{\beta d}{\alpha}\big)\Big)} \\
			&= \Tilde{C}^2_d\lambda^2\Gamma\Big(1-\frac{\beta d}{\alpha}\Big)\leftidx{_d}{\mathcal{G}}{^2_{\alpha,\beta}}(t,x)\text{E}_{1-\frac{\beta d }{\alpha},\ 1-\frac{\beta d}{\alpha}}\bigg(\leftidx{_d}{\Xi}{_{\alpha,\beta}}\lambda^2 \Tilde{C}^2_d\Gamma\Big(1-\frac{\beta d}{\alpha}\Big)t^{1-\frac{\beta d}{\alpha}}\bigg).
		\end{align*}

		Thus \eqref{Eq:3.19d} follows by setting 
		\begin{align*}
			\Psi_d:= \leftidx{_d}{\Xi}{_{\alpha,\beta}} \Tilde{C}^2_d\lambda^2\Gamma\Big(1-\frac{\beta d}{\alpha}\Big) \ \ 
			\text{and} \ \
			\leftidx{_d}{C}{_*}:=\lambda^2\Gamma\Big(1-\frac{\beta d}{\alpha}\Big)\Tilde{C}^2_d.
		\end{align*}

		\vspace{0.25cm}
		The bound \eqref{Eq:3.20d} easily follows from  \eqref{Eq:3.19d} as follows:
		\begin{align*}
			\Big(1\star\mathcal{K}_d\Big)(t,x)=& \int\limits_0^t \int\limits_{\mathbb{R}^d}  \mathcal{K}_d(s,y) dyds\\
			\geq & C\int\limits_0^t  \text{E}_{1-\frac{\beta d}{\alpha},\ 1-\frac{\beta d}{\alpha}}\Big(\Psi_d s^{1-\frac{\beta d}{\alpha}}\Big) \int\limits_{\mathbb{R}^d} \leftidx{_d}{\mathcal{G}}{^2_{\alpha,\beta}}(s,y) dyds. 
		\end{align*}
		Next,
		\begin{align*}
			\int\limits_{\mathbb{R}^d} \leftidx{_d}{\mathcal{G}}{^2_{\alpha,\beta}}(s,y)dy&=C_{d,\alpha}^2 s^{-\frac{\beta d}{\alpha}} \int\limits_{\mathbb{R}^d} \frac{1}{(1+|z|^2)^{( d+\alpha)}} dz\\
			&=\leftidx{_d}{C}{_\alpha} s^{-\frac{\beta d}{\alpha}}.
		\end{align*}
		Finally, using Eq. \eqref{IntMtgLf}, we conclude the proof.
	\end{proof}
	\vspace{0.5em}
	
	\begin{lemma}\label{lm:5.4-d}
		Suppose $\alpha\in (1,2]$ and $\mu\in \mathcal{M}_{\alpha, d}(\mathbb{R}^d), \ \mu\neq 0$. Then for all $\epsilon>0,$ there exists a constant $C$ such that for all $t\geq 0$ and $x\in \mathbb{R}^d,$
		\begin{equation}\label{lwb:J0}
			J_0(t,x)=\Big(G_d(t, \cdot)*\mu\Big)(x)\geq \leftidx{_d}{C}{_{\#}}\boldsymbol{1}_{\{t\geq \epsilon\}}	\leftidx{_d}{\mathcal{G}}{_{\alpha,\beta}}(t,x),
		\end{equation}
		where $\leftidx{_d}{C}{_{\#}}$ is a positive constant depending on $d, \alpha$ and $\beta.$
	\end{lemma}

	\begin{proof}
		The proof is very similar to the proof of Lemma \ref{lm:5.4} so we omit it.
	\end{proof}

	We can now state our main result in this section.
	
	\begin{theorem}\label{Thm:3.4-d}
		Suppose that $\alpha\in (1,2)$ and $\sigma$ satisfy the growth condition \eqref{eq:1.3}. For all $\mu\in \mathcal{M}_{\alpha, d, +}(\mathbb{R}^d), \mu\neq 0$ and for all $p\geq 2,$ if $\underline{\varsigma}=0,$ then 
		\begin{equation}
			\underline{\xi}(p)\geq \frac{\Psi_d^{\frac{1}{1-\frac{\beta d}{\alpha}}}}{2(\alpha+d)}>0,
		\end{equation}
		where $\Psi_d$ is as defined in Proposition \ref{Prop:3.3-d}. 
		
		For these $\mu,$ if $\underline{\varsigma}=0$ and $\mu(dx)=f(x)dx$ with $f(x)\geq c$ for all $x\in \mathbb{R}^d$ or if $\underline{\varsigma}\neq 0$, then $\underline{\xi}(p)=\bar{\xi}(p )=+\infty.$ 
		In particular, $\underline{\gamma}(p)=\bar{\gamma}
		(p)=+\infty.$ 
	\end{theorem}

	\begin{proof}
		The proof is also similar to the proof of Theorem \ref{Thm:3.4}, so we only provide sketches here. Again it is enough to only consider the case $p=2$ since $\underline{\xi}(p)\geq \underline{\xi}(2)$ for all $p\geq 2$.
		To this aim, fix $\epsilon\in(0,t/2)$ and choose a positive constant $\leftidx{_d}{C}{_{\#}}$ as in Lemma \ref{lm:5.4-d} such that 
		$$J_0(t,x)\geq \leftidx{_d}{C}{_{\#}}\boldsymbol{1}_{\{t\geq \epsilon\}}	\leftidx{_d}{\mathcal{G}}{^2_{\alpha,\beta}}(t,x):=\leftidx{_d}{I}{_\epsilon}(t,x).$$
		By an extension of  \cite[Eq. (3.4)]{ChenHuNua} we can show that,
		$$\norm{u(t,x)}_2^2\geq \Big(\leftidx{_d}{I}{^2_\epsilon}\star\underline{\mathcal{K}}_d\Big)(t,x).$$
		Now, applying Proposition \ref{Prop:3.3-d} and Lemma \ref{lwb-d} (2), we get
		\begin{align*}
			\Big(\leftidx{_d}{I}{_\epsilon^2}\star\underline{\mathcal{K}_d}\Big)(t,x)
			&\geq {C}_1 t^{-\frac{\beta d}{\alpha}}\int\limits_0^{t-\epsilon}  \text{E}_{1-\frac{\beta d}{\alpha}, \ 1-\frac{\beta d}{\alpha}}\big(\Psi_d s^{1-\frac{\beta d}{\alpha}}\big)s^{-\frac{\beta d}{\alpha}}(t-s)^{\frac{\beta d}{\alpha}}\leftidx{_d}{\mathcal{G}}{^2_{\alpha,\beta}}(t-s,x)ds.
		\end{align*}
		Since $\leftidx{_d}{\mathcal{G}}{_{\alpha,\beta}}(t-s,x)\geq \big(\frac{t-s}{t}\big)^\beta\leftidx{_d}{\mathcal{G}}{_{\alpha,\beta}}(t,x)$, we get
		\begin{align*}
			\Big(\leftidx{_d}{I}{_\epsilon^2}\star\underline{\mathcal{K}_d}\Big)(t,x)&\geq {C}_2t^{-\beta(2+d/\alpha)}\leftidx{_d}{\mathcal{G}}{^2_{\alpha,\beta}}(t,x)\int\limits_0^{t-\epsilon}  \text{E}_{1-\frac{\beta d}{\alpha}, \ 1-\frac{\beta d}{\alpha}}\big(\Psi_d s^{1-\frac{\beta d}{\alpha}}\big)s^{-\frac{\beta d}{\alpha}}(t-s)^{\beta(2+d/\alpha)} ds\\
			&\geq {C}_3t^{-\beta(2+d/\alpha)}\leftidx{_d}{\mathcal{G}}{^2_{\alpha,\beta}}(t,x) \text{E}_{1-\frac{\beta d}{\alpha}, \ 1-\frac{\beta d}{\alpha}}\Big(\Psi_d (t-2\epsilon)^{1-\frac{\beta d}{\alpha}}\Big)
			\int\limits_{t-2\epsilon}^{t-\epsilon} s^{-\frac{\beta d}{\alpha}}(t-s)^{\beta(2+d/\alpha)}ds\\
			&\geq  {C}_4t^{-\beta(2+d/\alpha)}\ \leftidx{_d}{\mathcal{G}}{^2_{\alpha,\beta}}(t,x) \text{E}_{1-\frac{\beta d}{\alpha}, \ 1-\frac{\beta d}{\alpha}}\Big(\Psi_d (t-2\epsilon)^{1-\frac{\beta d}{\alpha}}\Big) \frac{\epsilon^{\beta(2+d/\alpha)}}{(t-\epsilon)^{\frac{\beta d}{\alpha}}}\epsilon\\
			& =  C_{\epsilon}t^{-\beta(2+d/\alpha)}(t-\epsilon)^{-\frac{\beta d}{\alpha}}\leftidx{_d}{\mathcal{G}}{^2_{\alpha,\beta}}(t,x)\text{E}_{1-\frac{\beta d}{\alpha}, \ 1-\frac{\beta d}{\alpha}}\Big(\Psi_d (t-2\epsilon)^{1-\frac{\beta d}{\alpha}}\Big).
		\end{align*}
		Again, the function $x\mapsto\leftidx{_d}{\mathcal{G}}{^2_{\alpha,\beta}}(t,x)$ is even and decreasing for $|x|\geq 0.$ It follows that for all $a>0,$
		\begin{align*}
			\sup\limits_{|x|>\exp{(at)}}\norm{u(t,x)}_2^2\geq  C_{\epsilon}\  \leftidx{_d}{\mathcal{G}}{^2_{\alpha,\beta}}\big(t,\exp{(at)}\big)t^{-2\beta-\frac{\beta d}{\alpha}}(t-\epsilon)^{-\frac{\beta d}{\alpha}}\text{E}_{1-\frac{\beta d}{\alpha}, \ 1-\frac{\beta d}{\alpha}}\Big(\Psi_d (t-2\epsilon)^{1-\frac{\beta d}{\alpha}}\Big).
		\end{align*}
	
		Since $\alpha, \beta>0,$ there exists some $t_0\geq 0$ such that for all $t\geq t_0,$  $t^{\beta/\alpha}\leq e^{at},$ so
		$$\leftidx{_d}{\mathcal{G}}{^2_{\alpha,\beta}}\big(t,\exp{(at)}\big)\geq \frac{\hat{C}_{d,\alpha} t^{2\beta}}{e^{2a(\alpha+d)t}}.$$
		Finally, the asymptotic expansion of the Mittag-Leffler function in Lemma \ref{lm4.2} shows that
		
		\begin{align*}
			\lim\limits_{t\rightarrow\infty}\frac{1}{t}\sup\limits_{|x|\geq \exp{(at)}}\log\norm{u(t,x)}_2^2\geq \Psi_d^{\frac{1}{1-\frac{\beta d}{\alpha}}}-2a(\alpha+d).
		\end{align*}
		Therefore,
		\begin{align*}
			\underline{\xi}(2)=&\sup\Big\{a: \lim\limits_{t\rightarrow\infty}\frac{1}{t}\sup\limits_{|x|\geq \exp{(at)}}\log\norm{u(t,x)}_2^2>0 \Big\}\\
			\geq & \sup\Big\{a>0: \Psi_d^{\frac{1}{1-\frac{\beta d}{\alpha}}}-2a(\alpha+d)>0\Big\}\\
			=& \frac{\Psi_d^{\frac{1}{1-\frac{\beta d}{\alpha}}}}{2(\alpha+d)}.
		\end{align*}
		This concludes the first part of the Theorem.
		As for the second part, suppose that $\underline{\varsigma}=0$ and that there is $c>0$ such that $J_0\geq c,$ or that $\underline{\varsigma}\neq 0.$ In this case, by \cite[(3.4)]{ChenHuNua} and Proposition \ref{Prop:3.3-d}, we have 
		$$\norm{u(t,x)}_2^2\geq \max(c^2,\underline{\varsigma}^2)\big(1\star\underline{\mathcal{K}}_d\big)(t,x)\geq \bar{C}t^{1-\frac{\beta d}{\alpha}}\text{E}_{1-\frac{\beta d}{\alpha}, \ 2-\frac{\beta d}{\alpha}}\big(\Psi_d t^{1-\frac{\beta d}{\alpha}}\big).$$
		Again, since the lower bound above is independent of $x$, by Lemma \ref{lm4.2}, it follows that

		\begin{align*}
			\lim\limits_{t\rightarrow\infty}\frac{1}{t}\sup\limits_{|x|\geq \exp{(at)}}\log\norm{u(t,x}_2^2\geq \Psi_d^{\frac{1}{1-\frac{\beta d}{\alpha}}}.
		\end{align*}
		
		Thus, $\underline{\xi}(2)=+\infty$ and
		this concludes the proof.

	\end{proof}
	
	\begin{remark}
		We were not able to prove the inequality $\overline{\xi}(p)<\infty$  for $d>1$ at this moment. When $d>1$, the heat kernel does not have a uniform upper bound,  see for example  \cite[Page 7]{asogwa-foon-nane-2020}, making calculations extremely challenging. We leave this as an open problem for a future project.  
		
	\end{remark}

	\newpage
	\section{Appendix}\label{Appdx}
	
	\begin{lemma}\cite[Lemma 5.2]{ChenDalang}\label{lm52}
		Let $f_{b,\nu}(x):=f(x)=(b^2+x^2)^{-\nu-1/2} $ with $b>0$ and $\nu\geq 1/2.$ Then
		\begin{equation}\label{FTf}
			\mathcal{F}[f](z)=\int\limits_{\mathbb{R}} dxe^{-iz\cdot x}f(x)\geq C_\nu b^{-2\nu} e^{-b|z|},
		\end{equation}
		for all $b>0$ and $z\in\mathbb{R},$ where  $C_\nu=\frac{\Gamma(\nu)\Gamma(1/2)}{2\Gamma(\nu+1/2)}, \ \nu\geq 1/2.$
	\end{lemma}			
	
	\begin{lemma}\cite[Lemma 5.5]{ChenDalang}\label{lm5.5}
		Suppose $\nu>1.$ For all $x\in\mathbb{R},$ 
		$$\min\limits_{y\in\mathbb{R}}\big(|x-y|^\nu+|y|\big)\geq 
		\begin{cases}
			&\nu^{\frac{\nu}{1-\nu}}+\big||x|-\nu^{\frac{\nu}{1-\nu}}\big| \qquad \text{if} \ \ |x|> \nu^{\frac{\nu}{1-\nu}},\\
			&|x|^\nu \qquad\qquad\qquad\qquad\ \ \text{otherwise}
		\end{cases}$$ 
	\end{lemma}	
	
	\vspace{0.5cm}
	
	The (two-parameter) {\it Mittag-Leffler} function is defined as 
	
	\begin{equation}\label{MitgLf}
		E_{\mu,\nu}(z)=\sum\limits_{k=0}^\infty\frac{z^n}{\Gamma(\mu k+\nu)} \qquad \text{for} \ \ \mu>0, \ \nu>0..
	\end{equation}
	It has the following asymptotic expansion:
	
	\begin{lemma}\cite[Theorem 1.3]{Podlny}\label{lm4.2}
		If $0<\mu<2, \  \nu$ is an arbitrary complex number and $\upsilon$ is an arbitrary real number such that $\pi\mu/2<\upsilon<\pi\wedge (\pi\mu),$ then for any arbitrary integer $p\geq 1, $  the following expression holds:
		\begin{equation*}
			E_{\mu,\nu}(z)=\frac{1}{\mu}z^{(1-\nu)/\mu}\exp(z^{1/\mu})-\sum\limits_{k=1}^p\frac{z^{-k}}{\Gamma(\nu-\mu k)} +O\big(|z|^{-1-p}\big),
		\end{equation*}
		as $|z|\rightarrow\infty$ with $|\text{arg}(z)|\leq \upsilon.$ 
	\end{lemma}
	The following integration formula about the Mittag-Leffler function can also be found in \cite[(1.99) on P.24]{Podlny}
	\begin{equation}\label{IntMtgLf}
		\int\limits_0^zE_{\alpha,\beta}\big(\lambda t^{\alpha}\big)t^{\beta-1}dt=z^\beta E_{\alpha,\beta+1}\big(\lambda z^\alpha\big), \ \beta>0.
	\end{equation}

	The {\it Euler's Beta} function is defined as

	\begin{equation}\label{EulBta}
		\int\limits_0^t s^{a-1}(t-s)^{b-1} ds=\frac{\Gamma(a)\Gamma(b)}{\Gamma(a+b)}t^{a+b-1}, \ \text{Re}(a)>0, \text{Re}(b)>0. 
	\end{equation}
	
	Following \cite[5.12.3, p.142]{Handbk}, it can also be defined as:
	\begin{equation}\label{EulBta}
		\int\limits_0^\infty \frac{t^{a-1}}{(1+t)^{a+b}}dt=\frac{\Gamma(a)\Gamma(b)}{\Gamma(a+b)}.
	\end{equation}

\end{document}